\documentclass[11pt]{article}
\usepackage{graphicx}
\usepackage[colorlinks=True]{hyperref}
\usepackage{amsmath,amssymb,amsthm}
\usepackage{xcolor}
\usepackage{dsfont}
\usepackage{tikz,mathrsfs}
\usepackage{authblk}
\usepackage{geometry}
\usepackage{yhmath}

\usepackage{lipsum}

\newcommand\blfootnote[1]{%
  \begingroup
  \renewcommand\thefootnote{}\footnote{#1}%
  \addtocounter{footnote}{-1}%
  \endgroup
}
\usepackage{todonotes}
\newcommand{\cb}[1]{\textcolor{purple}{#1}}

\newcommand{\R}{\mathbb{R}}

\newcommand{\N}{\mathbb{N}}
\newcommand{\Z}{\mathbb{Z}}

\newcommand{\di}{{m_i}}
\newcommand{\lp}{{m}}

\renewcommand{\geq}{\geqslant}
\renewcommand{\leq}{\leqslant}

\usepackage{tikz}
\usetikzlibrary{arrows}
\usetikzlibrary{calc}

\DeclareMathOperator\supp{supp}

\DeclareMathOperator\conv{Conv} %cb, added, not sure it is the expected form

\newtheorem{thm}{Theorem}
\newtheorem{assump}{Assumption}
\newtheorem{defi}[thm]{Definition}
\newtheorem{lemma}[thm]{Lemma}
\newtheorem{prop}[thm]{Proposition}
\newtheorem{model}{Model}

\DeclareMathOperator{\Tr}{Tr}

\title{Limit shape of the leaky Abelian sandpile model with multiple layers}
\author[1]{Théo Ballu}

\author[2,3]{Cédric Boutillier}

\author[4]{Sevak Mkrtchyan}

\author[1,5]{Kilian Raschel}
\affil[1]{Univ Angers, CNRS, LAREMA, SFR MATHSTIC, F-49000 Angers, France}
\affil[2]{Sorbonne Université, CNRS, Laboratoire de Probabilités, Statistique et Modélisation,
LPSM, UMR 8001, F-75005 Paris, France}
\affil[3]{Institut Universitaire de France}
\affil[4]{University of Rochester, Department of Mathematics, Rochester, NY, United States}
\affil[5]{CNRS, International Research Laboratory France-Vietnam in Mathematics and its Applications, Vietnam Institute for Advanced Study in Mathematics, Hanoï, Vietnam}

\begin{document}

\date{\today}

\maketitle

\begin{abstract}
In this paper we study a triple generalization of the Leaky Abelian Sandpile Model (LASM)\ of Alevy and Mkrtchyan, originally analyzed in the case of the square lattice in dimension two. First, we work in any dimension. Second, each site can hold several different stacks of sand, one for each of a certain given number of different layers or colors. Third, when a stack of one color at a site topples, it can send sand not only to its nearest neighbors in equal amounts, but to all possible locations and colors, according to a fixed but arbitrary mass distribution. Stacks of different colors can topple according to different distributions and different leakiness parameters, however the toppling rule should be site-independent. We obtain three main results. First, in this generality, when the LASM is started with $N$ grains of sand in one color at the origin, the final stable configuration, after scaling down by $\log N$, converges to a limit shape as $N$ goes to infinity. Second, when the leakiness parameter converges to infinity and the toppling distribution has finite range, the limit shape converges to a polytope. Third, when the leakiness parameters converge to one, which means the leakiness disappears, the limit shape of the sandpile converges to an ellipsoid. From a technical point of view, we rely on a strong relation between the Green function for random walk and the shape of the sandpile. Finally, the limit shape exhibits interesting duality properties, which we also investigate.
\end{abstract}

\section{Introduction and main results}

The Abelian Sandpile Model (ASM)\ is a model introduced by Bak, Tang and Wiesenfeld
in 1987~\cite{BTW-ASM1987} as a ``toy model'' for studying self-organized criticality. In its basic form it can be thought of as a cellular automaton evolving on an arbitrary graph $G=(V,E)$, where $V$ is the set of vertices and $E$ the set of edges. A sandpile configuration on $G$ is a function $s:V\to\mathbb{Z}_{\geq 0}$. Given a configuration $s$, it evolves in discrete time according to the following principle: if the size $s(x)$ of the sandpile at a site $x$ is at least as large as the vertex degree of $x$, then $x$ ``topples'', sending one grain of sand to each neighboring vertex. It was observed by Bak, Tang and Wiesenfeld that this model organizes itself into a critical state in the following sense. If the model is started with a random configuration on a bounded portion of the square lattice $\mathbb{Z}^2$ and evolved until it stabilizes, the final stable configuration is critical, which means: if a small random addition of sand is made at a location, this change propagates throughout the lattice at all scales. 

It is thus of interest to understand what can be said of this final, critical configuration. More precisely, what can be said about the final configuration as $N$ grows
if the sandpile is started with the delta initial condition of $N$ grains of sand in one location? Simulations suggest that the region visited by the sandpile converges to a limit shape, and the final configuration has intricate local structure; see Figure~\ref{fig:three_examples}. While there is an understanding of the local structure in terms of Apollonian circle packings \cite{lps2016,ps2020}, few results about the limit shape have been obtained. It was shown in \cite{lp2009} that the final configuration is bounded between two circles. The existence of a limit shape was established in \cite{ps2013} while in \cite{as2019} it was shown that the shape is Lipschitz.

A modified, dissipative or leaky version of the model was considered by Richard Kenyon and Lionel Levine. On the square lattice $\mathbb{Z}^2$ the leaky model was studied in \cite{AlMk-22}. A new parameter $1\leq\lp\in\mathbb{R}$ was introduced, which controls how leaky the system is. The function $s$ now takes nonnegative real values and if a site $x$ has at least $4\lp$ grains of sand, then it topples, sending $1$ grain of sand to each neighbor, with the remaining $4(\lp-1)$ grains disappearing from the system. The quantity $1-\frac{1}{\lp}$ can naturally be interpreted as a mass (for an underlying massive Laplacian operator), so the case $\lp=1$ is sometimes called massless. In the massless case $\lp=1$ this specializes to the original ASM. In \cite{AlMk-22} the precise limit shape of the uniform nearest neighbor Leaky Abelian Sandpile Model (LASM) on the $2$-dimensional square lattice was computed for any $\lp>1$. It was also shown, that the limit shape converges to an $L_1$ ball as $\lp\to\infty$ and to a circle as $\lp\to 1$. It is worth noting, that even though the limit shape as $\lp\to 1$ converges to a circle, the limit shape when $\lp=1$, i.e.\ in the case of the original ASM, is not expected to be a circle, so the limit shape, as a function of $\lp$ should be discontinuous at $1$. 

In the present paper we study the LASM in a vastly more general setup. First, instead of working in dimension $2$, we work in arbitrary dimension. Second, each site can hold several different stacks of sand, one for each of $p$ different layers or colors. Third, when a stack of one color at a site topples, it can send sand not only to the nearest neighbors in equal amounts, but to all possible locations and colors, and according to a fixed but arbitrary mass distribution. Stacks of different colors can topple according to different distributions and different leakiness parameters, however the toppling rule should be site-independent. The only restrictions we place are some irreducibility and aperiodicity conditions, to ensure that all sites can be reached in an aperiodic way, and a moment condition to ensure mass doesn't quickly escape to infinity (see Assumption~\ref{assump:main_assumptions}).

Our main result is a proof that, in this generality, when the LASM is started with $N$ grains of sand in one color at the origin, the final stable configuration, after scaling down by $\log N$, converges to a limit shape (Theorem~\ref{thm:limit_shape_d_fixed}). To state the result, we need to introduce the function $\Gamma^{-1}:\mathbb S^{d-1}\to\mathbb R^d$ as in Proposition~\ref{prop:homeo}; this function actually parametrizes the exponential decay of a Green function at infinity, see Proposition~\ref{prop:asymptotics_green}.
\begin{thm}[Theorem~\ref{thm:limit_shape_d_fixed}]\label{thm:limit_shape_d_fixed_intro}
In direction $u\in\mathbb S^{d-1}$, for fixed $N$, the shape of the final configuration of the sandpile lies between two radii called $r_{N, { u}}$ and $R_{N, { u}}$. We have \[ \lim_{N \to +\infty}\frac{r_{N,u}}{\log N} = \lim_{N \to +\infty} \frac{R_{N,u}}{\log N} = \frac{1}{\Gamma^{-1}(u) \cdot u}\] and these limits are uniform in $u$.
    This means that the limit shape of the LASM is delimited by the curve 
    \begin{equation}\label{eq:limit_shape_in_spheric_coordinates_intro}
        \mathcal C = \left\{ \frac{u}{\Gamma^{-1}(u) \cdot u}~|~u \in \mathbb{S}^{d-1}  \right\},
    \end{equation}
    and this limit holds for the Hausdorff distance (and in a stronger uniform sense, to be introduced in Definition~\ref{def:cv_sets_limit_shape}).
\end{thm}
See Figure~\ref{fig:three_examples} for various illustrations of Theorem~\ref{thm:limit_shape_d_fixed_intro}. Moreover, in Theorem~\ref{thm:limit_shape_d_fixed} we prove that the deviation of the unscaled final configuration from the limit shape scaled up by $\log N$ is only of constant order. However, the exact formula \eqref{eq:limit_shape_in_spheric_coordinates_intro} for the limit shape is not easy to compute explicitly, since it involves inverting the gradient of a given spectral radius, see Proposition~\ref{prop:homeo}. As a further result (Proposition~\ref{prop:dual_curves}), we identify the limit curve with the dual of a curve which is the level set of the spectral radius mentioned above. For example, if there is only one color ($p=1$), then the spectral radius reduces to the Laplace transform of the increments of a random walk, whose level sets are well understood. Back to the multicolor setting, as the boundary of a convex shape, it follows that the limit shape itself bounds a convex region. 

Note, that the fact that the sandpile is started with $N$ grains of sand in one location and one color is not essential. If the sandpile is started from finitely many locations, with different number of grains in each of those location and in any colors, and $N$ is the largest number of grains of sand in any one color and location, then the limit shape will be exactly the same as that in Theorem \ref{thm:limit_shape_d_fixed_intro}.

We study two limiting regimes. First, we allow the leakiness parameter to converge to infinity at the same rate in all colors and study the case when the toppling distribution has finite range. In this regime we show that the limit shape converges to a polytope (Theorem~\ref{thm:limit_shape_as_convex_hull_of_returning_walks}), which depends only on the support of the toppling distribution, and not the distribution itself. Let $\mathcal X$ be the polytope defined in \eqref{prop:polytopeBase}.
\begin{thm}[Theorem \ref{thm:limit_shape_as_convex_hull_of_returning_walks}] \label{thm:limit_shape_as_convex_hull_of_returning_walks_intro}
The rescaled shape of the sandpile, $(\log \lp)\cdot \mathcal{C}$, where $\mathcal{C}$ is defined in~\eqref{eq:limit_shape_in_spheric_coordinates_intro}, converges to the convex hull of $\mathcal{X}$ as $\lp$ goes to $+\infty$. The convergence is uniform in the sense of Definition~\ref{def:cv_sets_limit_shape}. In particular, in the uncolored case, the limit shape of the sandpile when $\lp\to\infty$ is the convex hull of the support of the toppling distribution. 
\end{thm}

See the second row of Figure~\ref{fig:three_examples} for an illustration of Theorem~\ref{thm:limit_shape_as_convex_hull_of_returning_walks_intro}. The limit polytope can be described as the set of points which have the same first-passage-time for the underlying walk of the toppling distribution (Theorem~\ref{thm:limit_shape_as_first_passage_front}). This limiting polytope is the equivalent of the $L_1$ ball that was observed in \cite{AlMk-22}.

The second limit we study is when the leakiness parameters converge to $1$, which means the leakiness disappears. We show that in this limit the limit shape of the sandpile converges to an ellipsoid (Theorem~\ref{thm:limit_shape_ellipsoid}). 
\begin{thm}[Theorem~\ref{thm:limit_shape_ellipsoid}]
\label{thm:limit_shape_ellipsoid_intro}
    If the drift of the model is zero (in a sense to be defined in Section~\ref{subsec:zero_mass}), then the rescaled shape of the sandpile, $\sqrt{\lp-1}\cdot  \mathcal{C}$, where $\mathcal{C}$ is defined in \eqref{eq:limit_shape_in_spheric_coordinates_intro}, converges to the ellipsoid \[ \left\{ s \in \R^d ~|~ 2 s^\intercal \sigma^{-1} s \leq 1 \right\} \] as $\lp$ tends to $1$, where $\sigma^{-1}$ can be interpreted as the inverse of a covariance matrix. The convergence is uniform in the sense of Definition~\ref{def:cv_sets_limit_shape}.
\end{thm}
Note, however, that this does not identify the limit shape of the model when the leakiness parameters are in fact equal to $1$, which corresponds to the ordinary, non-leaky ASM. It follows from our work that the limit shape as a function of the leakiness parameters is continuous away from $1$, but at $1$ it is expected to be discontinuous.

All our limit shape results hold for uniform convergence in spherical coordinates (Definition~\ref{def:cv_sets_limit_shape}), which by our Proposition~\ref{prop:link_cv_hausdorff} implies convergence in Hausdorff distance (recalled in Definition~\ref{def:Hausdorff_distance}).

To obtain the limit shape results, we establish a link between the limit shape of the LASM and the Green function of a related killed random walk, showing that the limit shape can be bounded between two level curves of the Green function, thus reducing the study of the limit shape to the study of the corresponding Green function. See our Proposition~\ref{prop:thresholds}, which proves that if the Green function is above (resp.\ below) a threshold, then a given site is (resp.\ is not) in the shape of the final configuration. This connection between the sandpile and the Green function is crucial for at least two reasons. First, from a technical point of view, we can use existing results on the asymptotics of the Green function. Second, the Green function is ubiquitous in random walk theory, and more generally in probability theory, so it has been studied using various approaches. In Section~\ref{sec:Green_function} we review some of the literature and mention five interpretations of the exponential decay, all of which are relevant to the sandpile model. 

To give some more details, by applying a Doob transformation, the killed random walk is transformed into a non-killed random walk, the asymptotics of whose Green function is classical \cite{NeSp-66} when $p=1$ and was recently obtained in \cite{Du-20,Ba-24} in the multicolor setting. The exponential decay of the Green function established in \cite{Ba-24} implies that the distance between the two level curves of the Green function that bound the boundary of the region visited by the sandpile is of constant order, whereas their scale is of logarithmic order in the number of the initial grains of sand, establishing the limit shape. 

Note, that in the regime when the leakiness converges to $1$ fast enough, the distance between the two level curves grows, becoming comparable to their sizes, making it impossible to establish a limit shape result. This is partly the justification why the link between these Green functions and the limit shape breaks in that regime, making the discontinuity mentioned above possible.

Instead of multiple colors, the model we study can be equivalently described in terms of a single-color model on a larger graph with a site-dependent toppling distribution. More precisely, consider a graph $\mathcal{G}$ on which $\Z^d$ acts freely, but not transitively, and for which the quotient $F=\mathcal{G}/\Z^d$ has a finite number of vertices. On $\mathcal{G}$ consider a vertex-dependent toppling distribution which is invariant under the action of $\Z^d$. Such a model is equivalent to a colored model on $\Z^d$ with the number of colors equal to the size of the fundamental domain $F$. Note, that adding colors to the model on $\mathcal{G}$ would not change anything, as such a model could still be equivalently formulated as a colored model on $\Z^d$ by just adding extra colors.

The model with $p$ colors on $\mathbb{Z}^d$ can also equivalently be described as a single-color model on the vertices of the graph $\mathbb{Z}^d\times\{1,\ldots,p\}$.
Also note, that if a lattice can be embedded in the square lattice, then our results extend to such a lattice as well. For example, our results apply to the $2$-dimensional hexagonal lattice as well.

In Appendix~\ref{app:amoebas} we connect some of our results to the geometry of an amoeba, which is naturally associated with the model. In particular, we show that the limit shape \eqref{eq:limit_shape_in_spheric_coordinates_intro} is the dual of the connected component of the complement of the amoeba.

\begin{figure}[t!]
    \centering
    \includegraphics[width=2.1cm]{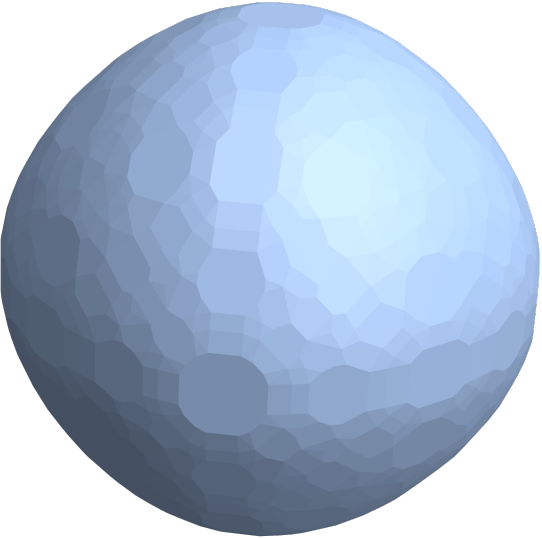}\qquad
    \includegraphics[width=2.1cm]{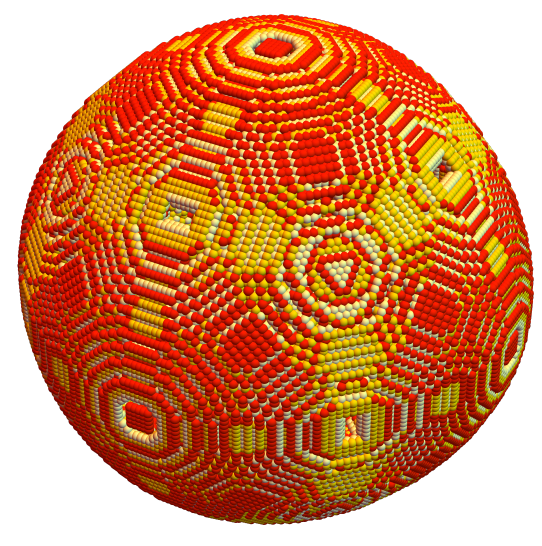}\qquad
    \includegraphics[width=8.4cm]{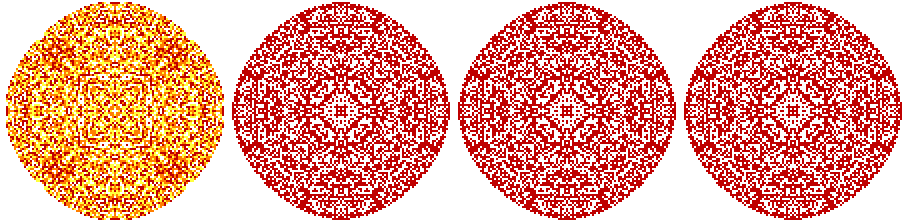}
    
    \includegraphics[width=2.1cm]{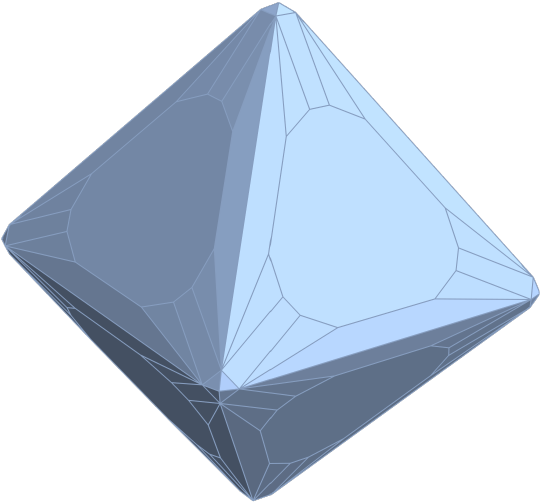}\qquad
    \includegraphics[width=2.1cm]{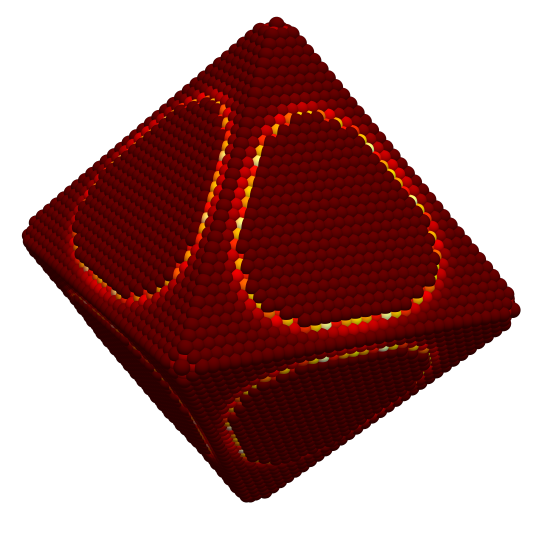}\qquad
    \includegraphics[width=8.4cm]{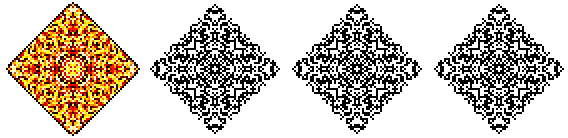}
    
    \includegraphics[width=2.1cm]{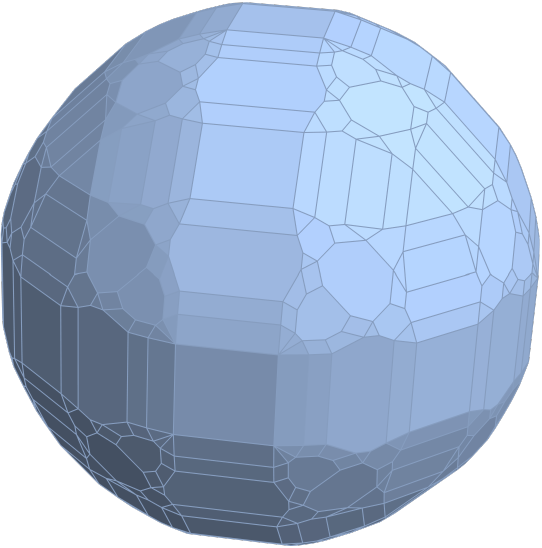}\qquad
    \includegraphics[width=2.1cm]{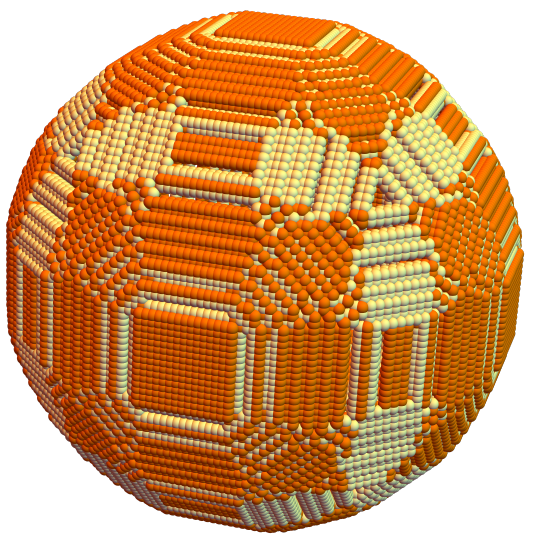}\qquad
    \includegraphics[width=8.4cm]{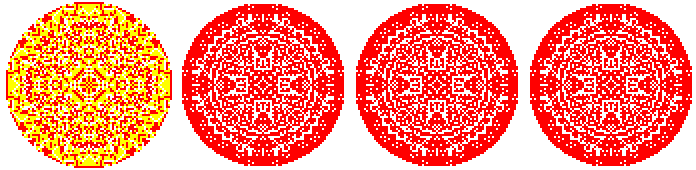}

    \caption{We consider the following sandpile model with four colors. Color $1$ sends $1$ chip to location $(0,0,0)$ to each of the colors $2,3,4$,
$m=4/3$ so $1$ chip is lost.
Color $2$ sends $2$ chips to color $1$, to the locations $(0,0,-1)$ and
$(0,0,1)$, $m=2$, so no chip is lost.
Color $3$ sends $2$ chips to color $1$, to the locations $(0,-1,0)$ and
$(0,1,0)$, $m=2$, so no chip is lost.
Color $4$ sends $2$ chips to color $1$, to the locations $(-1,0,0)$ and
$(1,0,0)$, $m=2$, so no chip is lost. For each of the three rows, the leftmost picture shows the convex hull of the final configuration; the second picture shows a unit sphere at each point of the boundary of the
visited region, the darker the sphere, the larger the height; the four right pictures show the middle horizontal slice of
the final configuration in each of the $4$ colors. 
The first row represents the above sandpile model with initial configuration of $10^{34}$ chips.
The second row represents the same sandpile with the exception that $m=10^8$ for
color $1$, and with initial configuration of $10^{300}$ chips. The last row represents the same sandpile with the exception that $m=3$ for color $1$,
so it is not a leaky model, and with initial configuration of $10^6$ chips.}
    \label{fig:three_examples}
\end{figure}

\section{Leaky sandpile model and the associated random walk}

In the following, $d,p,N \in \N := \{1,2,\ldots\}$. A \emph{sandpile configuration} is a function $s: \Z^d \times \{1, \ldots, p\} \to \R_{\geq 0}$.  To each $(x,i,j) \in \Z^d \times \{1, \ldots, p\}^2$, we associate $c(x,i,j) \geq 0$ and to each $i \in \{1, \ldots,p\}$, we associate a leakiness parameter $\di  \geq 1$. % with at least one $i \in \{1, \ldots, p\}$ such that $\di > 1$.

\begin{model}[Leaky Abelian Sandpile Model (LASM) on $\Z^d \times \{1, \ldots, p\}$]\label{model:sandpile}
    We assume that for all $i \in \{1, \ldots, p\}$, $\sum_{(y,j) \in \Z^d \times \{1, \ldots, p\}} c(y,i,j) < +\infty$.
    \begin{itemize}
        \item The initial configuration is $s_0 := N \delta_{\left(0,i_0\right)}$. It consists of $N$ grains of sand at the origin of $\Z^d \times \{1, \ldots,p\}$.
        \item A sandpile configuration $s$ is said to be \emph{stable} at site $(x,i)$ if: \[s(x,i) \leq \di \sum_{(y,j) \in \Z^d \times \{1, \ldots, p\}} c(y,i,j).\] We say a configuration is stable if it is stable at every site. Otherwise, we say it is unstable.
        \item As long as a configuration is unstable at site $(x,i)$, it can topple, which means:
        \begin{align*}
        &s(x,i) \leftarrow s(x,i) - \di \sum_{(y,j) \in \Z^d \times \{1, \ldots, p\}} c(y-x,i,j),\\
        &\forall (y,j) \in \Z^d \times \{1, \ldots,p\},\quad  s(y,j) \leftarrow s(y,j) + c(y-x, i, j).
        \end{align*}
        \item Sites topple until the configuration eventually becomes stable.
    \end{itemize}
\end{model}
Each time a site topples, $(\di-1) \sum_{(y,j) \in \Z^d \times \{1, \ldots, p\}} c(y,i,j)$ grains of sand disappear, hence the word \emph{leaky}. 
Starting from $s_0$, the configuration will eventually stabilize; in the case where all the $\di$ are strictly greater than $1$, there will be at most \begin{equation*}
   \frac{N}{ \min_{i \in \{1, \ldots,p\}}(\di-1)  \sum_{(y,j) \in \Z^d \times \{1, \ldots, p\}} c(y,i,j)}    
\end{equation*}
topplings. Besides, when several sites are unstable, the order in which they topple does not modify the stable configuration, hence the word \emph{Abelian}.

We define a killed random walk associated with the sandpile model. This random walk is a Markov-additive process on $\Z^d \times \{1, \ldots, p\}$. Markov-additive processes are Markov chains on the state-space $\Z^d \times E$ (here $E = \{1, \ldots,p\}$), whose jumps are invariant along the $\Z^d$ component, that is, a Markov chain $\left(X_n\right)_{n\geq 0}$ such that 
\begin{multline}\label{eq:Markov_additive_property}
\forall n \geq 0,\quad \forall i,j \in \{1, \ldots, p\},\quad  \forall x,y \in \Z^d,\\ \mathbb{P}\bigl( X_{n+1} = (y,j)| X_n = (x,i)  \bigr) = \mathbb{P}\bigl( X_{1} = (y-x,j)| X_0 = (0,i)  \bigr).
\end{multline}
Therefore, transitions of a Markov-additive process are fully described by the probabilities $\mu_{i,j}(x)$ to jump from $(0,i)$ to $(x,j)$ for $i,j \in \{1, \ldots, p\}$ and $x \in \Z^d$. In the case where $p=1$, the Markov-additive process is simply a sum of independent random variables with distribution $\mu = \mu_{1,1}$.
Generalities on Markov-additive processes and their applications to queuing theory can be found in Chapter~XI of \cite{As-03}.

\begin{model}[Killed Random Walk (KRW) associated with the sandpile model]\label{model:KRW}
    For $x \in \Z^d$ and $i,j \in \{1, \ldots, p\}$, let
    \begin{equation}
    \label{eq:def_mu_i,j}
        \mu_{i,j}(x) = \frac{c(x,i,j)}{\di \sum_{(y,j) \in \Z^d \times \{1, \ldots, p\}} c(y,i,j)}.
    \end{equation}
    We define a killed Markov-additive process $\left(X_n\right)_{n \geq 0}$ on $\left(\Z^d \times \{1, \ldots,p\}\right) \sqcup \{\mathsf{k}\}$ by the transitions
    \begin{align*}
    \mathbb{P}\big( (x,i) \to (y,j) \big) &= \mu_{i,j}(y-x) \\
    \mathbb{P}\big( (x,i) \to \mathsf{k} \big) &= 1- \sum_{(y,j) \in \Z^d \times \{1, \ldots,p\}} \mu_{i,j}(y) = 1 - \frac{1}{\di},
    \end{align*}
    where the state $\mathsf{k}$ is the cemetery state of the walk after it is killed.
\end{model}

The measure $\sum_{j=1}^p \mu_{i,j}$ is a probability measure if and only if $\di = 1$, that is, if there is no killing on color $i$. Otherwise, it is a sub-probability measure.

In order to apply previously known results on the Green function of Markov-additive processes, we will need the following assumptions on the Killed Random Walk (KRW).
\begin{assump}
\label{assump:main_assumptions}$ $
    \begin{enumerate}
        \item The KRW is killed at least on one color, that is, there exists $i \in \{1, \ldots, p\}$ such that $\di >1$.
        \item The KRW is irreducible, that is: 
        \begin{multline*}
\forall (x,i,j) \in \Z^d \times \{1, \ldots, p\}^2,\quad\exists n \in \N,\quad\exists (x_1,i_1), \ldots, \left(x_{n-1}, i_{n-1}\right) \in \Z^d \times \{1, \ldots, p\},\\ \mu_{i,i_1}(x_1) \mu_{i_1,i_2}(x_2 - x_1) \cdots \mu_{i_{n-1},j}(x-x_{n-1})> 0.
\end{multline*}
\item The KRW is aperiodic, that is:
\[\forall (x,i) \in \Z^d \times \{1, \ldots, p\},\quad \gcd \bigl\{ n\geq 1 ~|~ \mathbb{P}_{(x,i)}\bigl(X_n = (x,i)\bigr) >0 \bigr\} =1.\]
\item The KRW has finite exponential moments, that is:
\[\forall i,j \in \{1, \ldots, p\},\quad \forall \gamma >0,\quad \sum_{x \in \Z^d} e^{\gamma\|x\|} \mu_{i,j}(x) < +\infty.\]
    \end{enumerate}
\end{assump}

\section{Asymptotics of the Green function}
\label{sec:Green_function}

\subsection{The main result}

\begin{defi}
    The Green function of the KRW from Model~\ref{model:KRW} is 
    \begin{equation*}
        G\big( (x,i),(y,j) \big) = \sum_{n=0}^{+\infty} \mathbb{P}_{(x,i)} \big(X_n = (y,j) \big).
    \end{equation*}
\end{defi}

As a consequence of \eqref{eq:Markov_additive_property}, we have $G\big( (x,i),(y,j) \big) = G \big( (0,i), (y-x,j) \big)$, so in the following, we will only consider values of the Green function of the form $G\big( (0,i),(z,j) \big)$, where the space coordinate of the first variable is $0$.
%cb% Let $\Gamma$ be the function defined in Proposition~\ref{prop:homeo}.
The main objective of this section is to prove the following asymptotic result. 
For $x\in\mathbb R^d$ we denote by $\Vert x\Vert$ its Euclidean norm and by $\widehat x=\frac{x}{\Vert x\Vert}\in\mathbb S^{d-1}$ its direction.

\begin{prop}
\label{prop:asymptotics_green}
    There exist continuous functions $\chi_{i,j} : \mathbb{S}^{d-1} \to \R_{>0}$ such that, when $\|x\| \to + \infty$, \[ G\big( (0,i), (x,j)  \big) \sim \chi_{i,j}\left( \widehat x \right) e^{- \|x\| \Gamma^{-1} \left( \widehat x \right) \cdot \widehat x} \|x\|^{- \frac{d-1}{2}}, \]
    where $\Gamma$ is the function introduced in Proposition~\ref{prop:homeo}. 
\end{prop}
An important feature of the previous proposition is that, although these asymptotics depend on the direction $\widehat x$, they are uniform in $\widehat x$.

First of all we recall some relevant and inspiring literature.

\subsection{Literature on the Green function}
To provide some context for the asymptotics given in Proposition~\ref{prop:asymptotics_green}, and to introduce natural ideas and tools, we briefly review the literature on the asymptotics of the Green function. Let us start with the following general asymptotic statement of the Green function $G(x,y)=G(0,y-x)$ for random walks with drift on $\mathbb Z^d$: as $\Vert x\Vert\to\infty$,
\begin{equation}
    \label{eq:asymptotics_GF}
    G(0,x) \sim \chi(\widehat x)\exp\bigl( -\Vert x\Vert\nu(\widehat x) \bigr)\Vert x\Vert^{-\frac{d-1}{2}}.
\end{equation}
The quantity $\chi(\widehat x)$ is the prefactor and $\nu(\widehat x)\geq 0$ can be interpreted as the exponential decay to zero of the Green function. To illustrate the numerous results in the probabilistic and theoretical physics literature, let us mention five different interpretations of the exponential decay $\nu(\widehat x)$ and the underlying questions raised by these interpretations. In this paper we take inspiration from all of these perspectives.
When the Green function has such asymptotics, it is said to have an Ornstein-Zernike decay (see e.g.\ \cite{AoOtVe-24}).

\paragraph{A homeomorphism proved by Hennequin.}
In \cite[Thm~2.2]{NeSp-66}, Ney and Spitzer derive the asymptotics \eqref{eq:asymptotics_GF}. See \cite[Thm~2.4]{Do-66} for a similar result when $\Vert x\Vert\to\infty$ in the direction of the drift. In particular, Ney and Spitzer compute in \cite{NeSp-66} the exponential decay of the Green function when $\Vert x\Vert\to\infty$ in terms of a homeomorphism between the sphere and level sets of the Laplace transform of the increment distribution found by Hennequin \cite{He-63}. Since then, several alternative proofs of this asymptotic result \eqref{eq:asymptotics_GF} have been given, for example in \cite[Thm~25.15]{Wo-00}. 

In our more general context of walks on $\mathbb Z^d$ with $p$ colors or layers, i.e.\ $\mathbb Z^d\times \{1,\ldots,p\}$, the asymptotics of the Green function were obtained independently in \cite[Prop.~3.27]{Du-20} and \cite[Thm~1.5]{Ba-24}. Again, the exponential rate of decay to zero of the Green function is expressed in terms of a homeomorphism between the unit sphere and the level set of a given spectral radius. In all the above references, the asymptotics are derived in the case of zero mass (i.e., non-killed random walks). As we shall see, the classical Doob transformation can transfer the results in the presence of a non-zero mass (random walks with killing).

\paragraph{Second rate function in large deviation theory.}
In \cite[Cor.~5.7]{BoMo-96} it was shown that the exponential decay $\nu(\widehat x)$ in \eqref{eq:asymptotics_GF} can be reformulated in terms of the second rate function from large deviation theory. The latter is defined as follows: if $\Lambda$ denotes the classical rate deviation function associated to the random walk, then the second rate function is, for $x\in\mathbb Z^d$,
\begin{equation*}
    D(x) = \inf_{s>0} \frac{\Lambda(s x)}{s}.
\end{equation*}
This gives a different formulation and an interesting connection with large deviation theory. In the particular case of the simple random walk in dimension $2$, the expression of the exponential decay in terms of the second large deviation function was independently obtained in \cite[Prop.~8]{Me-06}.

\paragraph{Dual curves and Wulff crystals.}
Another equivalent formulation of the exponential decay $\nu(\widehat x)$ in \eqref{eq:asymptotics_GF} is provided in the paper \cite{BoMo-96}, see Theorem~1 there. It is proved that the function $D(x)$ above describes actually the dual set of the convex level set of the increment Laplace transform.

In the physics literature, the above dual set is often called a Wulff crystal or a Wulff shape. For instance, in the paper \cite{Me-06}, the author studies the surface tension of the Ising model at criticality. In the setting of the killed two-dimensional simple random walk, he derives the logarithmic asymptotics of the two-point correlation function (the Green function), see \cite[Thm~1]{Me-06}. In particular, he analyses the variation of the exponential decay of the Green function as a function of the killing rate, and studies the associated Wulff shapes. Wulff shapes actually arise in a variety of contexts in mathematical physics, such as Ising model and percolation theory, see \cite{Ce-06}.

\paragraph{Inverse correlation length.} In the paper \cite{AoOtVe-24} the authors consider the Green function of a large class of ``self-repulsive in average'' models, including the killed random walk and other self-repulsive polymer models, including in particular the self-avoiding walk. Under certain assumptions they prove that the Green function has Ornstein-Zernike behaviour, which means that the asymptotics of the Green function are given by \eqref{eq:asymptotics_GF}. In particular, the quantities $\chi$ and $\nu$ in \eqref{eq:asymptotics_GF} can be studied as a function of mass. The exponential decay $\nu(\widehat x)$ is interpreted as the inverse correlation length. Note that Wulff shapes also appear in \cite{AoOtVe-24}, not to describe the rate of convergence to zero, but as a tool in the pathwise analysis of the Green function.

\paragraph{Anisotropic norms.}
In \cite{MiSl-22} the authors consider the simple random walk in any dimension and obtain four different asymptotic regimes for the associated Green functions, depending on how (simultaneously)\ the point $x$ goes to infinity and the mass tends to $0$ or to $\infty$. One of the regimes corresponds to Ornstein-Zernike behaviour; the authors show that other relevant regimes appear. In \cite{MiSl-22} the quantity $x\mapsto \nu(\widehat x) \Vert x\Vert$ is interpreted as an anisotropic norm. A particularly interesting feature of the paper is the analysis of this norm as the mass goes to zero and infinity. In the zero-mass asymptotics the norm recovers an isotropy; on the contrary, in the infinite mass regime it converges to a ``non-smooth'' norm.

In dimension $2$, similar regimes are obtained in \cite[Thm~1.4]{AlMk-22} for the Green function of the killed simple random walk.

\paragraph{Lattice Green functions and differential equations.}

Although less related to asymptotic estimates, let us mention here a line of research started in the seventies, see e.g.\ \cite{Jo-73}, the main objective of which is to compute certain lattice Green functions. Indeed, as shown in \cite{Jo-73}, it is possible, for example, to compute the simple cubic lattice Green function in terms of certain special functions called Heun functions. More generally, there are several works which compute Green functions in closed form, see \cite{Gu-93,Gu-10}. These expressions can be hypergeometric functions, elliptic functions, Heun functions, classical integrals, diagonals of rational functions, etc. Sometimes they are expressed using their differential operators, and the need for computer algebra techniques can be particularly useful \cite{HaKoMaZe-16}. This line of research is still active, see e.g.\ 
\cite{ChChJi-24}. Although this has not yet been done, the use of analytic combinatorics in several variables techniques \cite{PeWi-13} is promising to deduce from these expressions their asymptotic behaviour.

\subsection{Proof of the main asymptotics result on the Green function}

In this section, we derive the asymptotics of the Green function $G$ stated in Proposition~\ref{prop:asymptotics_green} from the asymptotics of a \emph{non-killed} walk studied in \cite{Ba-24,Du-20}. To do so, we will perform a Doob transform to turn the killed random walk into a non-killed one with a drift.

The (real)\ Laplace transform of the walk is the function $L\mu : \R^d \to \mathcal{M}_p(\R)$ defined by:
\[  \forall t \in \R^d,\quad(L\mu)_{i,j}(t) :=  \sum_{x \in \Z^d} e^{t \cdot x} \mu_{i,j}(x).\]
Assumption~\ref{assump:main_assumptions} on the finite exponential moments ensures that it is well defined on the whole of $\R^d$.

\begin{lemma}[Lem.~2.10 in \cite{Ba-24}]\label{lem:c_doob_transf}
    Let $t \in \R^d$. We assume that $1$ is an eigenvalue of the Laplace transform $L\mu(t)$ with an eigenvector $\varphi_t$ whose all entries are positive. We define the measures $(\mu_t)_{i,j}$  by \[\forall i,j \in \{1, \ldots, p \},\quad\forall x \in \Z^d,\quad\left(\mu_t\right)_{i,j} (x) = \frac{\left(\varphi_t\right)_j}{ \left(\varphi_t\right)_i}e^{t \cdot x} \mu_{i,j}(x).\] Then for all $i \in \{1, \ldots, p\}$, $\sum_{j=1}^p  {(\mu_t)}_{i,j}$ is a \emph{probability} measure on $\Z^d$. This means that the Markov-additive process with transitions $\mu_t$ is a non-killed process.
\end{lemma}

The proof of Lemma~\ref{lem:c_doob_transf} is the same as that of \cite[Lem.~2.10]{Ba-24}. We briefly recall the main idea below to keep our paper self-contained.
In the work \cite{Re-24}, the Doob transform technique is applied to statistical mechanics by relating two models, namely random rooted spanning forests and random spanning trees, and providing various applications.

\begin{proof}
    The positivity of $\varphi_t$ ensures that $\sum_{j=1}^p  {(\mu_t)}_{i,j}$ is non-negative, and
\begin{equation*}
    \sum_{x\in \Z^d} \sum_{j=1}^p \left(\mu_t\right)_{i,j}(x)  = \sum_{j=1}^p \sum_{x \in \Z^d} \frac{\left(\varphi_t\right)_j}{\left(\varphi_t\right)_i} e^{ t \cdot x } \mu_{i,j}(x) = \frac{1}{\left(\varphi_t \right)_i} \left(L\mu(t) \varphi_t\right)_i = \frac{1}{\left(\varphi_t\right)_i} \left(\varphi_t\right)_i = 1.\qedhere
\end{equation*}
\end{proof}

Assumption~\ref{assump:main_assumptions} on irreducibility and aperiodicity ensures that $L\mu(t)$ is a primitive matrix for every $t \in \R^d$. Therefore, if the spectral radius of $L\mu(t)$ is $1$, then the Perron-Frobenius theorem ensures that there exists an eigenvector $\varphi_t$ like in Lemma~\ref{lem:c_doob_transf}, which is unique up to multiplication. This leads to the following definitions.
%textcolor{magenta}{Could we clarify a bit the above paragraph: first, do we mean non-negative or positive? second, what do we mean by ``such'' an eigenvector?}
%\textcolor{blue}{``such an eigenvector'' meant ``like in the previous lemma'', I clarified it. As for non-negative vs positive, I think the most precise answer is actually \emph{primitive}, that is, a non-negative, irreducible matrix, which has a power with only (strictly) positive entries. It follows from irreducibility+aperiodicity. For the Perron-Frobenius theorem, the only difference between \{non-negative and irreducible\} and primitive is that in the primitive case, there is uniqueness of the leading eigenvalue.}

\begin{defi}
\label{def:rho_T}
    Let $t \in \R^d$. The spectral radius of $L\mu(t)$ is denoted $\rho(t)$. We define, in $\R^d$, $\mathcal{T} := \rho^{-1}([0,1])$ and $\partial \mathcal{T} := \rho^{-1}(\{1\})$.
\end{defi}

\begin{defi}\label{def:doob_transform}
    Let $t \in \partial \mathcal{T}$ and $\varphi_t$ as in Lemma~\ref{lem:c_doob_transf}. The \emph{Doob transform} of parameter $t$ of the process ${(X_n)}_{n \geq 0}$ is the Markov-additive process $(X^t_n)_{n \geq 0}$ whose jump matrix is $\mu_t$ defined by 
    \[\forall i,j \in \{1, \ldots, p \},\quad\forall x \in \Z^d,\quad\left(\mu_t\right)_{i,j} (x) = \frac{\left(\varphi_t\right)_j}{ \left(\varphi_t\right)_i}e^{t \cdot x} \mu_{i,j}(x).\] It is a \emph{non-killed} process because Lemma~\ref{lem:c_doob_transf} ensures that every $\sum_{j=1}^p \left(\mu_t\right)_{i,j}$ is a probability measure, and no longer a sub-probability measure.
\end{defi}

\begin{lemma}
We have $\rho(0) < 1$, that is, $0 \in \mathring{\mathcal{T}}$.
\end{lemma}
In the proof below, inequalities between vectors are meant coordinate by coordinate.
\begin{proof}
Let $M = L\mu(0) = \bigl( \sum_{x \in \Z^d} \mu_{i,j}(x) \bigr)_{1 \leq i,j \leq p}$. It is a sub-stochastic matrix, so $\rho(M) \leq 1$, which leaves us to prove that $\rho(M) \neq 1$. At least one of the rows of $M$ has a sum strictly smaller than $1$, because  our running assumption implies that there exists $i \in \{1, \ldots, p \}$ such that $\di > 1$. If we denote by $\mathds{1}$ the vector of $\R^p$ with all coordinates equal to $1$, it means that $M\mathds{1} \leq \mathds{1}$ and $M \mathds{1} \neq \mathds{1}$. Let $x \in \R^d$ be a left eigenvector of $M$ associated with the eigenvalue $\rho(M)$, that is $x^\intercal M = \rho(M) x^\intercal$. The Perron-Frobenius theorem applied to $M^\intercal$ ensures that $x>0$. If $\rho(M) = 1$, we would have \[x^\intercal \left(M \mathds{1} - \mathds{1}\right) = x^\intercal \mathds{1} - x^\intercal \mathds{1} = 0.\]
But $M \mathds{1} - \mathds{1} \leq 0$ with at least one non-zero coefficient and $x >0$, so $x^\intercal \left(M \mathds{1} - \mathds{1}\right) < 0$, a contradiction.
\end{proof}

The Green function $G$ of the KRW and the Green function $G_t$ of the Doob transform of parameter~$t$ are related as follows.

\begin{prop}[Prop.~2.13 in \cite{Ba-24}]
\label{prop:link_green_functions}
    Let $t \in \partial \mathcal{T}$, $x \in \Z^d$ and $i,j \in \{1, \ldots, p\}$. The Green function of the Doob transform is given by 
    \[G_t \big( (0,i), (x,j) \big) = \frac{ \left(\varphi_t\right)_j}{\left(\varphi_t\right)_i} e^{t \cdot x} G\big( (0,i), (x,j) \big),\] where $\varphi_t$ is the Perron-Frobenius eigenvector introduced in Lemma~\ref{lem:c_doob_transf}.
\end{prop}
Proposition~\ref{prop:link_green_functions} is stated (and proved)\ in \cite{Ba-24} as Prop.~2.13. The main argument is briefly recalled below. 

\begin{proof}
Let $x \in \Z^d$, $i,j \in \{1, \ldots, p \}$ and $n\geq 0$.
    We have \[\mathbb{P}_{(0,i)}\bigl(X_n = (x,j)\bigr) = \sum_{(0,i) = (x_0, i_0) \to (x_1,i_1) \to \cdots \to (x_n,i_n) = (x,j)} \mu_{i,i_1}(x_1) \mu_{i_1, i_2}(x_2-x_1) \cdots \mu_{i_{n-1},j}(x-x_{n-1}),\] where the sum is taken on all paths from $(0,i)$ to $(x,j)$ of length $n$ in the graph associated with the Markov chain. Similarly, 
    \begin{align*}
     \mathbb{P}_{(0,i)}\bigl(X_n^t = (x,j)\bigr) &=  \sum_{  (0,i) \to (x_1,i_1) \to \cdots \to (x,j) } \frac{\left(\varphi_t\right)_{i_1}}{\left(\varphi_t\right)_{i}} e^{ t \cdot x_1 }\mu_{i,i_1}(x_1)   
     %\frac{\left(\varphi_t\right)_{i_2}}{\left(\varphi_t\right)_{i_1}} e^{ t \cdot (x_2 - x_1) }\mu_{i_1, i_2}(x_2-x_1) 
     \cdots  \frac{\left(\varphi_t\right)_{j}}{\left(\varphi_t\right)_{i_{ n-1 }}} e^{ t \cdot (x - x_{n-1}) }\mu_{i_{n-1},j}(x-x_{n-1}) \\
     &= \sum_{  (0,i) \to (x_1,i_1) \to \cdots \to (x,j) } \frac{\left(\varphi_t\right)_{j}}{\left(\varphi_t\right)_{i}} e^{ t \cdot x }  \mu_{i,i_1}(x_1) \mu_{i_1, i_2}(x_2-x_1) \cdots \mu_{i_{n-1},j}(x-x_{n-1})\\
     &= \frac{\left(\varphi_t\right)_{j}}{\left(\varphi_t\right)_{i}} e^{ t \cdot x } \mathbb{P}_{(0,i)}\bigl(X_n = (x,j)\bigr).
    \end{align*}
    Summing over $n \geq 0$ gives the result announced in Proposition~\ref{prop:link_green_functions}.
\end{proof}

In the following results we study the spectral radius $\rho$ in order to obtain useful properties of the set $\mathcal{T}$ introduced in Definition~\ref{def:rho_T}.

\begin{prop}[Prop.~2.14 in \cite{Ba-24}]\label{prop:usual_prop_spectral_radius}
The function $\rho$ is:
    \begin{itemize}
        \item convex (and even logarithmically convex);
        \item\label{it:prop:norm_coercivity} norm-coercive, i.e., $\rho(t) \xrightarrow[\|t\| \to + \infty]{}+\infty$.
    \end{itemize}
\end{prop}
We recall the proof given in \cite{Ba-24}, as we think it gives an interesting insight into the model. 

\begin{proof}
    The proof of the convexity relies on the result of \cite{Ki-61}. If $\mu_{i,j}(x) >0$, then $t \mapsto e^{t \cdot x} \mu_{i,j}(x)$ is logarithmically convex. Therefore, the entries of $L\mu$ are $0$ if $\mu_{i,j} = 0$, or logarithmically convex functions, as sums of logarithmically convex functions. Besides, by irreducibility and aperiodicity, for $n$ large enough, all entries of $\left(L\mu\right) ^n$ are positive.
        If all entries are logarithmically convex, then the result of \cite{Ki-61} ensures that $\rho$ is logarithmically convex. To deal with the $0$ entries, we take a closer look at the proof of \cite{Ki-61} and notice that if $\Tr \left( (L\mu)^n \right) > 0$ for $n$ large enough, which is the case, then the result holds.
        
We move to the proof of norm-coercivity. When a function is convex, its norm-coercivity is equivalent to coercivity in every direction.  Hence we only need to prove that for every $u \in \mathbb{S}^{d-1}$, $\rho(r u) \to+\infty$ as $r\to+\infty$. 
         According to the Perron-Frobenius theorem, see for example \cite[Cor.~1]{Se-06},
        \begin{equation*}
            \rho(ru)^n \geq \min_{i \in \{1, \ldots, p \}} \sum_{j=1}^p \left(L\mu (ru)^n \right)_{i,j},
        \end{equation*}
        so it is enough to prove that for every $u \in  \mathbb{S}^{d-1}$, there exists $n \in \N$ such that  for every $i,j \in \{1, \ldots, p \}$, $\left(L\mu (ru)^n \right)_{i,j} \to+\infty$ as $r \to + \infty$.
        Let $u \in \mathbb{S}^{d-1}$ and $i,j \in \{1, \ldots,p\}$. Let $x \in \Z^d$ such that $x \cdot u >0$. Then as a consequence of aperiodicity and irreducibility, there exists $n_{i,j}\in \N$ such that for every $n \geq n_{i,j}$, $\mathbb{P}_{(0,i)}\bigl(X_n = (x,j)\bigr) > 0$. Let $n = \displaystyle \max_{i,j \in \{1, \ldots, p \}} n_{i,j}$. Then \[\left(L\mu(ru)^n\right)_{i,j} \geq e^{r x \cdot u} \mathbb{P}_{(0,i)}\bigl(X_n = (x,j)\bigr) \xrightarrow[r \to +\infty]{}+\infty,\] which proves the norm-coercivity.
\end{proof}

\begin{prop}\label{prop:doob_non-centered}
    For every $t \in \partial \mathcal{T}$, $\nabla \rho (t) \neq 0$.
\end{prop}
\begin{proof}
    The function $\rho$ is convex, so if $\nabla \rho(t) = 0$ for some $t \in \partial \mathcal{T}$, then $\rho$ has a global minimum in $t$. But $\rho(t) = 1$ and $\rho(0) < 1$, hence a contradiction.
\end{proof}

The gradient of $\rho(t)$ is the right notion for the drift of the Doob transform of parameter $t$, see \cite[Def.~1.3, Prop.~2.16]{Ba-24}. Therefore, Proposition~\ref{prop:doob_non-centered} means that the Doob transform of parameter $t \in \partial \mathcal{T}$ is non-centered, allowing to use previously known results on non-killed walks with a drift.

\begin{prop}[Prop.~2.16 in \cite{Ba-24}]
\label{prop:pos_def}
    For every $t \in \partial \mathcal{T}$, the Hessian $\mathbf{H}_\rho(t)$ of $\rho$ is positive-definite.
\end{prop}

The next result is crucial to get the exponential decay of the Green function in Proposition~\ref{prop:asymptotics_green}. With the interpretation of the gradient of $\rho$ as the drift of the process, it means that the Doob transform allows us to change the drift in any direction of $\R^d$.

\begin{prop}[Thm~2.17 in \cite{Ba-24}]
\label{prop:homeo}
    The function \[ \Gamma: \begin{array}[t]{ccl}
         \partial \mathcal{T}& \longrightarrow & \mathbb{S}^{d-1}  \\
         t& \longmapsto &  \frac{1}{\|\nabla \rho(t)\|} \nabla \rho(t)
    \end{array}\] is a homeomorphism between $\partial \mathcal{T}$ and the sphere $\mathbb{S}^{d-1}$ of $\R^d$. 
\end{prop}

\begin{prop}
\label{prop:vraiment_expo}
    Let $u \in \mathbb{S}^{d-1}$. Then $\Gamma^{-1}(u) \cdot u > 0$.
\end{prop}

\begin{proof}
    It comes down to proving that for every $t \in \partial \mathcal{T}$, $t \cdot \nabla \rho(t) >0$.
  By convexity of $\rho$, for every $t \in \partial \mathcal{T}$, $\mathcal{T}$ is included in the affine half-space 
  \begin{equation*}
      H_t = \{x \in \R^d~|~(x-t) \cdot \nabla \rho(t) \leq 0\}.
  \end{equation*}
  Besides, since the Hessian is positive-definite by Proposition~\ref{prop:pos_def}, $\partial H_t \cap  \mathcal{T} = \{t\}$. But $t \neq 0$ because $0 \in \mathcal{T} \setminus \partial \mathcal{T}$, so $0 \notin \partial H_t \cap \mathcal{T}$, hence $0 \in H_t \setminus \partial H_t$, i.e., $t \cdot \nabla \rho(t) >0$.
\end{proof}

The previous proposition ensures that the decay of Proposition~\ref{prop:asymptotics_green} is actually exponential in every direction. This is not true for non-killed walks studied in \cite{Ba-24}, where the decay in the direction of the drift is only polynomial.

Now we can complete the proof of Proposition~\ref{prop:asymptotics_green}.

\begin{proof}[Proof of Proposition~\ref{prop:asymptotics_green}]
Let us fix $t_0 \in \partial\mathcal{T}$.  We remind that $G_{t_0}$ is the Green function of a \emph{non-killed} walk, with a non-zero drift. Therefore, under Assumption~\ref{assump:main_assumptions}, \cite[Thm~1.5]{Ba-24} shows that there exist continuous functions $\widetilde{\chi}_{i,j} : \mathbb{S}^{d-1} \to \R_{>0}$ and ${c : \mathbb{S}^{d-1} \to \R^d}$ such that \[G_{t_0} \big( (0,i), (x,j) \big) \sim \widetilde{\chi}_{i,j}\left( \frac{x}{\|x\|} \right) e^{- \|x\| c\left( \frac{x}{\|x\|} \right) \cdot \frac{x}{\|x\|} } \|x\|^{- \frac{d-1}{2}}.\]
Hence using  Proposition~\ref{prop:link_green_functions}
\begin{equation}
\label{eq:asymptotics_green_fixed_t0}
G\big( (0,i), (x,j) \big)= \frac{ \left(\varphi_{t_0}\right)_i}{\left(\varphi_{t_0}\right)_j} e^{-t_0 \cdot x} G_{t_0} \big( (0,i), (x,j) \big) \sim \frac{ \left(\varphi_{t_0}\right)_i}{\left(\varphi_{t_0}\right)_j} \widetilde{\chi}_{i,j}\left( \frac{x}{\|x\|} \right) e^{-\|x\| \left( t_0 + c\left( \frac{x}{\|x\|} \right) \right) \cdot \frac{x}{\|x\|}} \|x\|^{\frac{d-1}{2}}.
\end{equation}
It remains to identify the exponential decay constant $t_0 + c\bigl( \frac{x}{\|x\|} \bigr)$ in \eqref{eq:asymptotics_green_fixed_t0}. To do so we fix a direction $u \in \mathbb{S}^{d-1}$ and work with positive multiples of $u$, that is $x = \|x\| u$. We choose $t_u= \Gamma^{-1}\left(u \right)$, so that 
\begin{equation}
\label{eq:asymptotics_green_t_right_direction}
G\big( (0,i), (\|x\|u,j) \big) = \frac{ \left(\varphi_t\right)_i}{\left(\varphi_t\right)_j} e^{-\|x\| \Gamma^{-1}\left( u \right) \cdot u} G_{t_u}\big( (0,i), (\|x\|u,j) \big).
\end{equation}
We apply once again \cite[Thm~1.5]{Ba-24} to get the asymptotics of $G_{t_u} \big( (0,i), (\|x\|u,j) \big)$ as $\|x\| \to +\infty$. This time, our choice of $t_u$ makes the drift of the Doob transform in the direction of $u$. Therefore, there is no exponential decay in $G_{t_u} \big( (0,i), (\|x\|u,j) \big)$, which only decreases polynomially.
Identifying the exponential decay in \eqref{eq:asymptotics_green_fixed_t0} and \eqref{eq:asymptotics_green_t_right_direction}, we get $t_0 + c\left( \frac{x}{\|x\|} \right) = \Gamma^{-1} \left( \frac{x}{\|x\|} \right)$, and thus \eqref{eq:asymptotics_green_fixed_t0} gives the result.
\end{proof}

Let us make two remarks on Proposition~\ref{prop:asymptotics_green}. First, in the proof, we could not directly choose $t = t_u$ instead of a fixed $t_0$, because in that case, the Green function $G_{t_u}$ changes with the direction of $x$. Therefore, \cite[Thm~1.5]{Ba-24} would only yield the asymptotics in a fixed direction $u$, while we want the asymptotics to be uniform in $u$.

Moreover, using Proposition~\ref{prop:vraiment_expo} and a continuity-compactness argument, we have
\begin{equation}\label{eq:greenfct_minimal_decay}
\min_{u \in \mathbb{S}^{d-1}} \Gamma^{-1}(u) \cdot u >0.
\end{equation}
Denoting this minimum by $\gamma$, we see that the Green function decreases faster than $e^{- \gamma \|x\|}$.

\section{Odometer function%, related operators
and thresholds of the Green function}
The main goal of this section is to provide a link between the Green function of the killed random walk and the shape of the sandpile. In Proposition~\ref{prop:thresholds}, generalizing the idea of \cite{AlMk-22}, we obtain two thresholds for the Green function $G\big( (0,i_0), (x,i) \big)$: one above which the point $(x,i)$ is inside the sandpile and one under which it is outside the sandpile.  To do so, we introduce the odometer function, a classical function in sandpiles, and related operators.

\begin{defi}
    We define the following operator $T: \ell^\infty \left( \Z^d \times \{1, \ldots, p\} \right) \to \ell^\infty \left( \Z^d \times \{1, \ldots, p\} \right)$, which can be interpreted as a massive Laplacian:
    \begin{align}\label{eq:definition_T}
        (T v)(x,i) &= \sum_{(y,j) \in \Z^d \times \{1, \ldots, p\}} \dfrac{c(x-y,j,i)}{\lp_j \sum_{(z,k)\in \Z^d \times \{1, \ldots, p\}} c(z,j,k) }v(y,j) - v(x,i) \\
       &= \sum_{(y,j) \in \Z^d \times \{1, \ldots, p\}} \mathbb{P}\big( (y,j) \to (x,i)\big) v(y,j) - v(x,i).\label{eq:T_bien_def}
    \end{align}
       If we see $T$ as a square matrix with rows and columns indexed by $\Z^d \times \{1, \ldots,p\}$, we have
    \begin{align*}
        T\big( (x,i), (y,j) \big) = T \mathds{1}_{(y,j)} (x,i) &= \frac{c(x-y,j,i)}{\lp_j \sum_{(z,k)\in \Z^d \times \{1, \ldots, p\}} c(z,j,k) } - \mathds{1}_{(x,i) = (y,j)} \\&= \mathbb{P}\big( (y,j) \to (x,i)\big) - \mathds{1}_{(x,i) = (y,j)}.
    \end{align*}
\end{defi}

Let us check that the formula \eqref{eq:definition_T} for $Tv$ does indeed define an operator $\ell^\infty \left( \Z^d \times \{1, \ldots, p\} \right) \to \ell^\infty \left( \Z^d \times \{1, \ldots, p\} \right)$.
One has 
\begin{equation*}
    \sum_{(y,j) \in \Z^d \times \{1, \ldots, p\}} \left| \mathbb{P}\big( (y,j) \to (x,i)\big) v(y,j) \right| %&= \sum_{j=1}^p \sum_{y \in \Z^d }  \mathbb{P}\big( (0,j) \to (x-y,i)\big) |v(y,j)|\\
     \leq \|v\|_\infty \sum_{j=1}^p \sum_{y \in \Z^d }  \mathbb{P}\big( (0,j) \to (x-y,i)\big) 
    %&= \|v\|_\infty \sum_{j=1}^p \underbrace{\sum_{z \in \Z^d }  \mathbb{P}\big( (0,j) \to (z,i)\big) }_{\leq 1}\\
     \leq p \|v\|_\infty,
\end{equation*}
therefore $Tv$ is well defined in $\ell^\infty \left( \Z^d \times \{1, \ldots, p\} \right)$ and $\|Tv\|_\infty \leq (p+1) \|v\|_\infty$.

    \begin{defi}
        The transpose %$G^\intercal$ 
        of the Green function defines an operator $G^\intercal : \ell^\infty \left( \Z^d \times \{1, \ldots, p\} \right) \to \ell^\infty \left( \Z^d \times \{1, \ldots, p\} \right)$ by
        \begin{equation*}
            G^\intercal v(x,i) = \sum_{(y,j) \in \Z^d \times \{1, \ldots, p\}} G\big( (y,j), (x,i) \big)v(y,j).
        \end{equation*}
    \end{defi}

Let us check that $G^\intercal$ does indeed define an operator $\ell^\infty \left( \Z^d \times \{1, \ldots, p\} \right) \to \ell^\infty \left( \Z^d \times \{1, \ldots, p\} \right)$:
    \[\sum_{(y,j) \in \Z^d \times \{1, \ldots, p\}} | G\big( (y,j), (x,i) \big)v(y,j) |  \leq \|v\|_\infty \sum_{j=1}^p \sum_{z \in \Z^d} G\big( (0,j), (z,i) \big).\]
According to Proposition~\ref{prop:asymptotics_green} and to \eqref{eq:greenfct_minimal_decay}, there exist constants $\gamma, C >0$ such that for every $(z,i,j)$, $G\big( (0,j), (z,i) \big) \leq C e^{- \gamma \|z\|}$. Therefore,
\[ \sum_{(y,j) \in \Z^d \times \{1, \ldots, p\}} | G\big( (y,j), (x,i) \big)v(y,j) | \leq C p \|v\|_\infty \sum_{z \in \Z^d} e^{- \gamma \|z\|} < +\infty. \]
This shows that $G^\intercal v \in \ell^\infty \left( \Z^d \times \{1, \ldots, p\} \right)$ and $\|G^\intercal v \|_\infty \leq C p \|v\|_\infty \sum_{z \in \Z^d} e^{-\gamma \|z\|}$.

The following proposition expresses $T$ as an inverse of the Green function $G$.

\begin{prop}\label{prop:inverse_T_G}
    The products $G^\intercal T$ and $T G^\intercal$ are well defined and equal to the negative identity $-I_{\Z^d \times \{1, \ldots, p\}}$.
\end{prop}
    In particular, $T^{-1} = - G^\intercal$ shows that $T^{-1}$ has negative coefficients. Therefore, if $v \geq 0$, then $T^{-1}v \leq 0$.

\begin{proof}
The products are well defined because the columns of these matrices are in $\ell^\infty \left( \Z^d  \times \{1,\ldots,p\} \right)$. Indeed, for $T$ we only need to bound probabilities by $1$, and for $G^\intercal$, we use the bound $G\big( (0,j), (z,i) \big) \leq C e^{- \gamma \|z\|} \leq C$ for constants $\gamma, C>0$, which is a consequence of Proposition~\ref{prop:asymptotics_green}.

    By definition of $T$ and $G$, 
        \begin{equation}\label{eq:produit_T_G}
        \left(TG^\intercal\right)_{(x,i),(y,j)} = \sum_{(z,k) \in \Z^d \times \{1, \ldots,p\}} \mathbb{P}\big( (z,k) \to (x,i) \big) G \big( (y,j),(z,k) \big) - G\big( (y,j), (x,i) \big).
         \end{equation}
      Besides, conditioning according to $X_{n-1}$, we get:
        \begin{align*}
           G\big( (y,j), (x,i) \big) &=  \mathds{1}_{(y,j) = (x,i)} + \sum_{n=1}^{+ \infty} \mathbb{P}_{(y,j)}\bigl(X_n = (x,i) \bigr)\\
            &= \mathds{1}_{(y,j) = (x,i)} + \sum_{n=1}^{+ \infty} \sum_{(z,k) \in \Z^d \times \{1, \ldots,p\}} \mathbb{P}_{(y,j)}\bigl(X_n = (x,i) ~|~ X_{n-1}=(z,k) \bigr) \mathbb{P}_{(y,j)} \bigl( X_{n-1}=(z,k) \bigr)\\
            &= \mathds{1}_{(y,j) = (x,i)} + \sum_{(z,k) \in \Z^d \times \{1, \ldots,p\}} \mathbb{P}\bigl( (z,k) \to (x,i) \bigr) \sum_{m=0}^{+\infty}\mathbb{P}_{(y,j)} \bigl( X_{m}=(z,k) \bigr)\\
            &= \mathds{1}_{(y,j) = (x,i)} + \sum_{(z,k) \in \Z^d \times \{1, \ldots,p\}} \mathbb{P}\bigl( (z,k) \to (x,i) \bigr) G \bigl( (y,j),(z,k) \bigr).
        \end{align*}
        Putting it into \eqref{eq:produit_T_G}, we get the expected $\left(TG^\intercal\right)_{(x,i),(y,j)} = - \mathds{1}_{(y,j) = (x,i)}$.

    Similarly, 
    \begin{equation}\label{eq:produit_G_T}
        \left(G^\intercal T \right)_{(x,i),(y,j)} = \sum_{(z,k) \in \Z^d \times \{1, \ldots,p\}}  G \bigl( (z,k),(x,i) \bigr) \mathbb{P} \bigl( (y,j) \to (z,k) \bigr) - G\bigl( (y,j), (x,i) \bigr).
         \end{equation}
    Besides, conditioning according to $X_{1}$, we get:
    \begin{align*}
        G\bigl( (y,j), (x,i) \bigr) &= \mathds{1}_{(y,j) = (x,i)} + \sum_{n=1}^{+ \infty} \mathbb{P}_{(y,j)}\bigl(X_n = (x,i) \bigr)\\
        &= \mathds{1}_{(y,j) = (x,i)} + \sum_{n=1}^{+ \infty} \sum_{(z,k) \in \Z^d \times \{1, \ldots,p\}} \mathbb{P}_{(y,j)}\bigl(X_n = (x,i) ~|~ X_1 = (z,k) \bigr) \mathbb{P}_{(y,j)}\bigl(X_1 = (z,k)\bigr)\\
        &= \mathds{1}_{(y,j) = (x,i)} + \sum_{(z,k) \in \Z^d \times \{1, \ldots,p\}} \mathbb{P}\bigl( (y,j) \to (z,k) \bigr) \sum_{m = 0}^{+\infty} \mathbb{P}_{(z,k)}\bigl(X_m= (x,i)\bigr)\\
        &= \mathds{1}_{(y,j) = (x,i)} + \sum_{(z,k) \in \Z^d \times \{1, \ldots,p\}} \mathbb{P}\bigl( (y,j) \to (z,k) \bigr) G \bigl( (z,k),(x,i) \bigr).
    \end{align*}
    Putting it into \eqref{eq:produit_G_T}, we get the expected $\left(G^\intercal T\right)_{(x,i),(y,j)} = - \mathds{1}_{(y,j) = (x,i)}$.
    \end{proof}

    \begin{defi}[\cite{AlMk-22}]
        The \emph{odometer} function is the function $u : \Z^d \times \{1, \ldots,p\} \to \R$ defined by \[u(x,i) = \text{total mass of sand emitted from } (x,i) \text{ before stabilization}.\]
        We say that a site $(x,i)$ is in the shape of the final configuration if $u(x,i) >0$.
    \end{defi}
    The odometer function counts the grains sent to the neighbours \emph{and} those that disappeared because of the leakiness. Note that it implicitly depends on the initial configuration.
    
    One might notice that in the shape of the final configuration, we do not count the sites that received grains of sand but did not emit any. Note, that those sites are at distance at most $1$ (in the graph associated with the Markov chain) of sites that were counted, so in the case when the support of $\mu$ is finite, adding them will not modify the limit shape of the sandpile as $N$ goes to infinity. Furthermore, it might happen that sites received sand during the stabilization process but sent it all and have an amount of sand equal to zero in the final configuration. With our definition, those sites have a positive odometer and thus belong to the final shape.

    The following proposition links the operator $T$ and the odometer function to the final configuration of the sandpile. It is a generalization of \cite[Eq.~(3.2)]{AlMk-22} to our context.
    \begin{prop}\label{prop:T_applied_to_odometer}
        Let $s_0$ be any initial configuration with a finite number of grains of sand, and denote by $f$ the final configuration obtained after stabilization.
        Then 
        \begin{equation*}
            Tu = f - s_0.
        \end{equation*}
        In particular, if we start the sandpile with $N$ grains at $(0,i_0)$, then
        \begin{equation*}
            Tu = f - N \delta_{(0,i_0)}.
        \end{equation*}
    \end{prop}
    \begin{proof}
        According to \eqref{eq:definition_T}, \[(T u)(x,i) = \sum_{(y,j) \in \Z^d \times \{1, \ldots, p\}} \dfrac{c(x-y,j,i)}{\lp_j \sum_{(z,k)\in \Z^d \times \{1, \ldots, p\}} c(z,j,k) } u(y,j) - u(x,i).\]
        The definition~\ref{model:sandpile} of the LASM shows that when the site $(y,j)$ topples, the proportion of sand leaving $(y,j)$ that is sent to $(x,i)$ is \begin{equation*}
            \dfrac{c(x-y,j,i)}{\lp_j \sum_{(z,k)\in \Z^d \times \{1, \ldots, p\}} c(z,j,k) }.
        \end{equation*}
        Therefore, the total mass received by $(x,i)$ is \begin{equation*}
            \sum_{(y,j) \in \Z^d \times \{1, \ldots, p\}} \dfrac{c(x-y,j,i)}{\lp_j \sum_{(z,k)\in \Z^d \times \{1, \ldots, p\}} c(z,j,k) } u(y,j).
        \end{equation*} Hence 
        \begin{align*}
        (Tu)(x,i) &= \text{total mass received by $(x,i)$}-\text{total mass emitted from $(x,i)$}\\
        &= f(x,i) - s_0(x,i).\qedhere
        \end{align*}
    \end{proof}

    The result below gives the connection between the level sets of the Green function and the final configuration. It is an adaptation of \cite[Lem.~3.3]{AlMk-22} to our general setting.
    \begin{prop}\label{prop:thresholds} There exist constants $\alpha, \beta > 0$ such that:
        \begin{itemize}
            \item if $G\big( (0,i_0), (x,i) \big) > \frac{\alpha}{N}$, then $(x,i)$ is in the shape of the final configuration;
            \item if $G\big( (0,i_0), (x,i) \big) < \frac{\beta}{N}$, then $(x,i)$ is not in the shape of the final configuration.
        \end{itemize}
    \end{prop}
    
    \begin{proof}
        By Propositions~\ref{prop:inverse_T_G} and~\ref{prop:T_applied_to_odometer},
        \[T\left( -G\big( (0,i_0), \cdot \big) + \frac{1}{N} u \right)= \frac{1}{N} f.\] The finite family $\left( \di \sum_{(y,j)} c(y,i,j) \right)_{i \in \{1, \ldots,p\}}$ is bounded by a constant $a$, and in the final configuration, no site has more than $a$ grains of sand. Therefore, 
        \[ 0 \leq T\left( -G\big( (0,i_0), \cdot \big) + \frac{1}{N} u \right) \leq \frac{a}{N} \mathds{1} \] where $\mathds{1}$ is the constant function equal to $1$.
Multiplying by $T^{-1} = -G^\intercal$, which has negative coefficients, we get \[0 \geq \frac{1}{N}u(x,i) - G\big( (0,i_0), (x,i) \big) \geq -\frac{a}{N}  G^\intercal \mathds{1} (x,i).\]
         Moreover, \[G^\intercal \mathds{1} (x,i) = \sum_{(y,j)} G\big( (y,j),(x,i) \big) = \sum_{j=1}^p \sum_{z \in \Z^d} G\big( (0,j),(z,i) \big),\]
        and this sum is finite because of the exponential decay of the Green function. Moreover, it does not depend on $x$ and there is a finite number of $i \in \{1, \ldots, p\}$, so it is bounded uniformly in $(x,i)$, hence the existence of a constant $\alpha$ such that:
        \begin{equation}\label{eq:encadrement}
            0 \geq \frac{1}{N}u(x,i) - G\big( (0,i_0), (x,i) \big) \geq -\frac{\alpha}{N}.
        \end{equation}
        If $(x,i)$ is not in the shape of the final configuration, then $u(x,i) = 0$ and from \eqref{eq:encadrement}, we get \[G\big( (0,i_0), (x,i) \big) \leq \frac{\alpha}{N},\]
        whereas if it is in the
        %pile
        \cb{shape}, then $(x,i)$ has emitted sand at least once, thus $u(x,i) \geq \di  \sum_{(y,j)} c(y,i,j)$ and \eqref{eq:encadrement} leads to \[G\big( (0,i_0), (x,i) \big) \geq \frac{1}{N}u(x,i) \geq \frac{1}{N} \min_{i \in \{1, \ldots,p\}} \di \sum_{(y,j)} c(y,i,j),\] so we can choose the constant $\beta := \displaystyle \min_{i \in \{1, \ldots,p\}} \di \sum_{(y,j)} c(y,i,j)$.
    \end{proof}

    \section{Limit shape %of the sandpile
    when the number of grains tends to infinity}

    In this section, we prove Theorem~\ref{thm:limit_shape_d_fixed_intro}, which is our main result. The existence of a limit shape for the stable configuration as the number of grains of sand $N$ goes to infinity, after normalization by $\log N$, is proved. We give an explicit formula for the limit shape that uses the homeomorphism $\Gamma$ of Proposition~\ref{prop:homeo} and also describe it as the dual curve of the set $\partial \mathcal{T}$ introduced in Definition~\ref{def:rho_T}.
    
    The general idea is to use the exponential decay of the Green function from Proposition~\ref{prop:asymptotics_green} to show that the two level sets of Proposition~\ref{prop:thresholds} are close enough, so that after rescaling by $\log N$, they converge to the same curve.
    
    \subsection{Notions of convergence for sets}
    
    We first define the notion of convergence of sets that we use. Given that the exponential decay of the Green function depends on the direction, we use convergence in spherical coordinates.
    
    \begin{defi}
    \label{def:cv_sets_limit_shape}
    Let $\left( E_N \right)_{N\geq 0}$ be a sequence of subsets of $\R^d$ and $\mathcal{C}$ a curve which can be described in spherical coordinates as \[\mathcal{C} = \left\{ r_u u ~|~u \in \mathbb{S}^{d-1} \right\},\]
    with $r\colon u \in\mathbb{S}^{d-1} \mapsto r_u\in\R_+$.
    We say that the limit shape of the sets $E_N$ is delimited by $\mathcal{C}$ if, for every $u \in \mathbb{S}^{d-1}$, there exist two sequences $\left( r_{u,N} \right)_{N\geq 0}$ and $\left( R_{u,N} \right)_{N\geq 0}$ such that, for every $N \in \N$, 
    \[\left\{ r u ~|~u \in \mathbb{S}^{d-1}, r \leq r_{u,N} \right\} \subset E_N \subset \left\{ R u ~|~u \in \mathbb{S}^{d-1}, R \leq R_{u,N} \right\}\]
    and for every $u \in \mathbb{S}^{d-1}$, 
    \begin{equation*}
        r_{u,N} \xrightarrow[N \to +\infty]{}r_u\quad \text{and}\quad
        R_{u,N} \xrightarrow[N \to +\infty]{}r_u.
    \end{equation*}
    If these two limits are uniform in $u \in \mathbb{S}^{d-1}$, that is
    \begin{equation*}
        \sup_{u \in \mathbb{S}^{d-1}} \left| r_{u,N} - r_u \right|\xrightarrow[N \to +\infty]{}0\quad \text{and}\quad
        \sup_{u \in \mathbb{S}^{d-1}} \left| R_{u,N} - r_u \right|\xrightarrow[N \to +\infty]{}0,
    \end{equation*}
    we say that the convergence is uniform.
    \end{defi}
    
    %In the special case where the sets $E_N$ are star-shaped at $0$, they have a form
    %\[E_N = \left\{ r u ~|~u \in \mathbb{S}^{d-1}, 0 \leq r \leq r_{u,N} \right\},\]
    %that is, they are delimited by the curves 
    %\[\mathcal{C}_N := \left\{ r_{u,N} u ~|~u \in \mathbb{S}^{d-1}\right\},\] the convergence of the sets $E_N$ comes down to proving that for every $u \in \mathbb{S}^{d-1}$, $r_{u,N} \xrightarrow[N \to +\infty]{}r_u$.

    We now relate the convergence of sets from Definition~\ref{def:cv_sets_limit_shape} to the classical Hausdorff distance convergence, which we first recall.
    
    \begin{defi}\label{def:Hausdorff_distance}
    Let $X,Y \subset \R^d$ be non-empty. The Hausdorff distance between $X$ and $Y$ is 
    \begin{equation*}
        d_H(X,Y) = \max \left( \sup_{x\in X}d(x,Y), \sup_{y \in Y}d(y,X) \right),
    \end{equation*}
    where $d(a,B) = \inf_{b \in B}\|a-b\|$.
    \end{defi}
    The Hausdorff distance restricted to closed, bounded, non-empty subsets of $\R^d$ defines a distance in the sense of metric spaces.
    
    \begin{prop}
    \label{prop:link_cv_hausdorff}$ $
        \begin{enumerate}
            \item \label{prop:link_cv_hausdorff:it:uniform} If a sequence of sets $\left(E_N\right)_{N\geq 0}$ converges \emph{uniformly} to
            \[\mathcal{C} = \left\{ r_u u~|~u \in \mathbb{S}^{d-1} \right\}\]
            in the sense of Definition~\ref{def:cv_sets_limit_shape}, then we have $d_H\left( E_N,E \right) \xrightarrow[N \to +\infty]{}0$ where \[E = \left\{ r u ~|~u \in \mathbb{S}^{d-1}, 0 \leq r \leq r_u \right\}\]is the set delimited by the curve $\mathcal{C}$.

            \item The conclusion is not necessarily true if the convergence to the limit shape is not uniform.
            
            \item The converse of the assertion~\ref{prop:link_cv_hausdorff:it:uniform} is false; Hausdorff convergence does not even imply convergence in the non-uniform sense of Definition~\ref{def:cv_sets_limit_shape}.
        \end{enumerate}
    \end{prop}
    
    \begin{proof}
    Let us assume the uniform convergence with the notations of Definition~\ref{def:cv_sets_limit_shape}.
    Let $x = ru \in E_N$ with $r\geq 0$ and $u \in \mathbb{S}^{d-1}$. If $r \leq r_u$, then $d(x,E) = 0$; whereas if $r \geq r_u$, then $r \leq R_{u,N}$, so \[d(x,E) \leq d(r u,r_u u) = |r-r_u| \leq R_{u,N} - r_u.\] Therefore, \[\sup_{x \in E_N}d(x,E) \leq \max \left( 0, \sup_{u \in \mathbb{S}^{d-1}}\bigl(R_{u,N} - r_u\bigr)\right) \leq \sup_{u \in \mathbb{S}^{d-1}} \left| R_{u,N} - r_u \right|.\]
     Similarly, let $y = r u \in E$, which means $u \in \mathbb{S}^{d-1}$ and $r \leq r_u$. If $r \leq r_{u,N}$, then $y \in E_N$ so $d(y,E_N) = 0$; whereas if $r > r_{u,N}$, then $d(y,E_N) \leq d(y, r_{u,N}u) = r- r_{u,N} \leq r_u - r_{u,N}$. Therefore, 
     \[\sup_{y \in E_N}d\left(y,E_N\right) \leq \max \left( 0, \sup_{u \in \mathbb{S}^{d-1}}\bigl(r_u - r_{u,N}\bigr)\right) \leq \sup_{u \in \mathbb{S}^{d-1}} \left|r_u - r_{u,N}\right| .\]
     In conclusion,\[d_H(E_N,E) \leq \max \left(\sup_{u \in \mathbb{S}^{d-1}} \left|R_{u,N} - r_u\right|, \sup_{u \in \mathbb{S}^{d-1}} \left|r_u - r_{u,N}\right| \right) \xrightarrow[N \to +\infty]{}0.\]
    
    We now give a counterexample in the non-uniform case. Let $\left(u_N\right)_{N\geq 0}$ be an injective sequence of $\mathbb{S}^{d-1}$ and \[E_N = \bar{B}(0,1) \cup [0,2u_N]\]
    the closed unit ball pricked with a toothpick at $u_N$. Then, with notations from Definition~\ref{def:cv_sets_limit_shape}, we can choose $R_{u,N} = r_{u,N} = 1$ if $u \neq u_N$ and $R_{u_N,N} = r_{u_N,N} = 2$. The sequence $\left(u_N\right)_{N\geq 0}$ being injective, for fixed $u$, we have $R_{u, N} = r_{u,N} = 1$ for $N$ large enough. Therefore, $\left(E_N\right)_{N\geq 0}$ converges (non-uniformly)\ to the unit ball $\bar{B}(0,1)$ in the sense of Definition~\ref{def:cv_sets_limit_shape}. However, $d_H\left(E_N,\bar{B}(0,1)\right) = 1$ for every $N \in \N$.
    
    We finally give an example of convergence with respect to the Hausdorff distance where there is no convergence in the sense of Definition~\ref{def:cv_sets_limit_shape}.
    Let $E_N$ be a circular sector as in Figure~\ref{fig:circular_sector} with $\theta_N = \frac{\pi}{N}$. Then it is easy to see that $d_H\left(E_N, \bar{B}(0,1) \right) \leq \frac{\pi}{2N}$, so $E_N$ converges to the unit ball for the Hausdorff distance, but for every $N \in \N$, the radius in direction $(1,0)$ is $0$ which does not converge to $1$.
    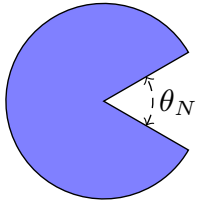
\begin{figure}
    \begin{center}
        \scalebox{1.3}{%
        \begin{tikzpicture}
        \draw[fill opacity=0.5,fill=blue] (0,0) -- (0.866,0.5) arc (30:330:1) --cycle ;
        \draw[densely dashed, <->] (0.433,-0.25) arc (-30:30:0.5) node[anchor=north west]{\footnotesize{$\theta_N$}} ;
        \end{tikzpicture}}
    \end{center}
        \caption{Sets that converge to the unit ball for the Hausdorff distance, but not in the sense of Definition~\ref{def:cv_sets_limit_shape}.}\label{fig:circular_sector}
    \end{figure}
    \end{proof}
    
    \subsection{Convergence of the sandpile to a limit shape}
    Let us fix $i_0, j \in \{1, \ldots, p\}$. In order to use some tools from classical analysis, we first extend continuously the function $x  \mapsto G\big( (0,i_0), (x,j) \big)$ from $\Z^d$ to $\R^d$,
    by saying that the value at $x\in\R^d$ is a convex combination of the values at neighboring lattice points.
    To do so, we define $\theta :[0,+\infty)\to [0,1]$ by
    \begin{equation*}
      \theta(s)=
      \begin{cases}
        1-s &\text{if $s\in[0,1]$},\\
        0&\text{if $s>1$},
      \end{cases}
    \end{equation*}
    and \[G\big( (0,i_0), (x,j) \big) = \dfrac{\sum_{y \in \Z^d}\theta\left( \|x-y\|_\infty \right) G\big( (0,i_0), (y,j) \big)}{\sum_{y \in \Z^d}\theta\left( \|x-y\|_\infty \right)}.\]
    We will denote $u \in \mathbb{S}^{d-1}$ and
    \begin{equation}
      \label{eq:green_fixed_direction}
      g_{u} : r \mapsto G \big( (0,i_0), (ru,j) \big)
    \end{equation}
    the Green function in direction $u$. Note that the function $(u,r) \in \mathbb{S}^{d-1} \times [0,+\infty) \mapsto g_u(r)$ is continuous.

    \begin{defi}
      We define radii $r_{N, u}$ and $R_{N, u}$ by:
      \begin{equation*}
        r_{N, u} = \inf \left\{r>0~|~g_u(r) \leq \frac{\alpha}{N} \right\}\quad\text{and}\quad
        R_{N, u} = \sup \left\{ r>0~|~ g_u(r) \geq \frac{\beta}{N} \right\}.
      \end{equation*}
      In other words, using the following generalized inverses
      \begin{equation*}
        g_{1,u}^{-1} : y \mapsto \inf \left\{ r>0 ~|~ g_u(r) \leq y \right\}\quad \text{and}\quad 
        g_{2, { u}}^{-1} : y \mapsto \sup \left\{ r>0~|~ g_u(r) \geq y \right\},
      \end{equation*} we have
      \begin{equation*}
        r_{N, { u}} = g_{1, { u}}^{-1}\left( \frac{\alpha}{N} \right) \quad \text{and}\quad 
        R_{N, { u}} = g_{2, { u}}^{-1}\left( \frac{\beta}{N} \right).
      \end{equation*}
    \end{defi}

    The generalized inverses of $g_{{ u}}$ are well defined, at least for $y>0$ small enough, because $g_{ { u}}(r)\to0$ as $r \to +\infty$. Besides, by continuity of $g_{ { u}}$, for $i \in \{1,2\}$,
    \begin{equation}\label{eq:composition_generalized_inverse}
        g_{{ u}}\left( g_{i, { u}}^{-1} (y) \right) = y.
    \end{equation}
    
The need for these definitions of the radii comes from the fact that $r \mapsto G\bigl((0,i_0),(ru,j)\bigr)$ is not necessarily decreasing (think of a model with long jumps and no short jumps), hence the use of generalized inverses. Besides, in order to use Proposition~\ref{prop:thresholds}, we cannot use the same generalized inverse for both radii.
    
    If $r < r_{N, { u}}$, then $G \big( (0,i_0), (ru,j) \big) > \frac{\alpha}{N}$ and if $r>R_{N, { u}}$, then $G \big( (0,i_0), (ru,j) \big) < \frac{\beta}{N}$. Therefore, Proposition~\ref{prop:thresholds} shows the following.
    \begin{prop}
        In direction $u$, for fixed $N$, the shape of the final configuration of the sandpile lies between radii $r_{N, { u}}$ and $R_{N, { u}}$.
    \end{prop}

    To get a limit shape after normalization by $\log N$, we need to prove that $\frac{r_{N, { u}}}{\log N}$ and $\frac{R_{N, { u}}}{\log N}$ have a common limit, uniformly in $u$.
     According to Proposition~\ref{prop:asymptotics_green}, there exist constants $C_{u}, \gamma_{u}$ (where $C_{u}$ depends on $i_0$, $j$ and $u$; $\gamma_{u} = \Gamma^{-1}(u) \cdot u$ only depends on $u$, and both are continuous functions of $u$ and positive by Proposition~\ref{prop:vraiment_expo})\ such that, when $r$ goes to infinity,
    \begin{equation*}
    %\label{eq:dev_asympt_G_fixed_direction}
        g_u(r) = C_{u} e^{-\gamma_{u} r} r^{- \frac{d-1}{2}} (1 + o(1)).
    \end{equation*}
    Besides, these asymptotics are uniform in $u$, which means that there is a function $\varepsilon$ that does not depend on $u$ such that $\varepsilon(r) \xrightarrow[r \to +\infty]{}0$ and $|o(1)|\leq \varepsilon(r)$. We will exploit these asymptotics to get the asymptotics of $g_{u,i}^{-1}$, and thus of the radii $r_{N,u}$ and $R_{N,u}$.
    
    We will denote by $f_{u}$ the asymptotic equivalent of $g_u$, that is:
    \begin{equation*}
        f_{u} (r):= C_{u} e^{-\gamma_{u} r} r^{- \frac{d-1}{2}}, \quad r > 0.
    \end{equation*}
    Although the function $g_u$ is not necessarily invertible, the function $f_u$ is, and its inverse is easily expressed with the classical Lambert $W$ function, which satisfies $W(y)e^{W(y)}=y$.
    
    \begin{prop}\label{prop:asymptotics_of_inverse_equivalent}
    Writing $p=\frac{2}{d-1}>0$ and $f_{u} : r \mapsto C_{u} e^{-\gamma_{u} r} r^{- \frac{1}{p}}$, one has
    \begin{equation*}
        f_{u}^{-1}(y) = \frac{1}{p\gamma_{u}} W\left( p\gamma_{u} C_u^{p}y^{-p} \right).
    \end{equation*}
    \end{prop}

    We remind the classical asymptotics of the function $W$ as $x \to +\infty$, see e.g.\ \cite[Eq.~(4.18)]{CoGoHaJeKn-96}:
    \begin{equation}\label{eq:asymptotics_W}
        W(x) = \log x - \log \log x + o(1).
    \end{equation}

%\textcolor{magenta}{Quick question: is $p=\frac{1}{\delta}$ always? I would suggest to mention clearly that $p>0$}
   \begin{prop}\label{prop:asympotics_inverse_equivalent_uniform}
We note $f_u^{-1}(y) = a_u W \left(b_u y^{-p}\right)$, where $a_u$ and $b_u$ are positive and continuous functions of $u$. Then, when $y \to 0$,
    \begin{equation*}
        f_u^{-1}(y) = a_u \left[ p \log \left( \frac{1}{y} \right) - \log \log \left( \frac{1}{y} \right) + \log\left( \frac{b_u}{p} \right) \right] + o(1)
    \end{equation*}
    and these asymptotics are uniform in $u$, that is, there exists a function $\varepsilon$ independent of $u$ such that $\varepsilon(y) \xrightarrow[y \to 0]{}0$ and $|o(1)| \leq \varepsilon(y)$.
\end{prop}

\begin{proof}
        Using \eqref{eq:asymptotics_W}, there exists a function $\epsilon$ such that $\epsilon(x) \xrightarrow[x \to +\infty]{}0$ and such that $W(x) = \log x - \log \log x + \epsilon(x)$.
        Therefore,
        \begin{align*}
            W\left(b_u y^{-p} \right) &= \log\left(b_u y^{-p} \right) - \log \log \left(b_u y^{-p} \right) + \epsilon\left(b_u y^{-p} \right)\\
            %&= \log (b_u) +p \log \left( y^{-1}\right) - \log \left( \log (b_u) +p \log\left( y^{-1} \right) \right) + \epsilon\left(b_u y^{-p} \right)\\
            &= p \log \left( y^{-1}\right) + \log(b_u) - \log\left( p \log\left( y^{-1} \right)   \left( 1+ \frac{\log(b_u)}{p \log\left( y^{-1} \right) } \right)\right) + \epsilon\left(b_u y^{-p} \right)\\
            &= p \log \left( y^{-1}\right) + \log(b_u) - \log(p) - \log \log\left( y^{-1} \right) - \log\left( 1+ \frac{\log(b_u)}{p \log\left( y^{-1} \right) } \right) +\epsilon\left(b_u y^{-p} \right) .
        \end{align*}
We now prove that $\log\left( 1+ \frac{\log(b_u)}{p \log\left( y^{-1} \right) } \right) +\epsilon\left(b_u y^{-p} \right)$ tends to $0$, uniformly in $u$.
        Since $b_u$ is a positive continuous function of $u$, and $u$ lies in the compact set $\mathbb{S}^{d-1}$, there exist constants $m,M >0$ such that $\forall u \in \mathbb{S}^{d-1}$, $m \leq b_u \leq M$. Therefore, $m y^{-p} \leq b_u y^{-p}$, hence \[\left|\epsilon\left( b_u y^{-p} \right)\right| \leq \sup_{x \geq my^{-p}} \epsilon\left(x \right) \xrightarrow[y \to 0]{}0,\]
        where the bound $ \sup_{x \geq my^{-p}} \epsilon\left(x \right)$ no longer depends on $u$.
        Moreover, $\log (m) \leq \log (b_u) \leq \log (M)$, so \[\left| \log\left( 1+ \frac{\log(b_u)}{p \log\left( y^{-1} \right) } \right) \right| \leq  \max \left( \left| \log\left( 1 + \frac{\log(m)}{p\log\left(y^{-1}\right)} \right) \right|, \left| \log\left( 1 + \frac{\log(M)}{p\log\left(y^{-1}\right)} \right) \right| \right) \xrightarrow[y \to 0]{} 0.\]
        In conclusion, if we define
        \[\varepsilon(y) = \left(\max_{u \in \mathbb{S}^{d-1}} a_u\right) \left[\sup_{x \geq my^{-p}} \epsilon\left( x \right) + \max \left( \left| \log\left( 1 + \frac{\log(m)}{p\log\left(y^{-1}\right)} \right) \right|, \left| \log\left( 1 + \frac{\log(M)}{p\log\left(y^{-1}\right)} \right) \right| \right)\right],\]
        we have $\varepsilon(y) \xrightarrow[y \to 0]{}0$ and \[\left|f_u^{-1}(y) - a_u \left[ p \log \left( \frac{1}{y} \right) - \log \log \left( \frac{1}{y} \right) + \log\left( \frac{b_u}{p} \right) \right]\right| \leq \varepsilon(y).\qedhere\]
    \end{proof}
    
    \begin{lemma}\label{lem:conv_infty_uniform}
    For $i \in \{1,2\}$, $g_{i,u}^{-1}(y) \xrightarrow[y \to 0]{}+\infty$ uniformly in $u$, that is: 
    \begin{equation}
        \label{eq:toprove_conv_infty_uniform}
        \forall A >0,\quad \exists y_0 >0,\quad\forall u \in \mathbb{S}^{d-1},\quad\forall y \in (0,y_0),\quad g_{i,u}^{-1}(y) \geq A.
    \end{equation}
\end{lemma}

\begin{proof}
        Since $g_u = f_u \cdot ( 1+o(1))$ at $+\infty$, uniformly in $u$, there exists $r_0 >0$ such that for every $r \geq r_0$ and every $u \in \mathbb{S}^{d-1}$, $\frac{1}{2} f_u(r) \leq g_u(r) \leq 2 f_u(r)$.
        
        Let $A > r_0$. We define \[y_1 := \frac{1}{2}\min \left\{ g_u(r)~|~0 \leq r \leq r_0, u \in \mathbb{S}^{d-1} \right\}.\] Note that $y_1$ is well defined by continuity of $(r,u) \mapsto g_u(r)$ and compactness of $\bar{B}(0,r_0)  \times \mathbb{S}^{d-1}$. We also set \[y_2:= \frac{1}{2} \min_{u\in \mathbb{S}^{d-1}} f_u(A),\] which is well defined because $u \mapsto  f_u(A)$ is continuous and $\mathbb{S}^{d-1}$ is compact. We have $y_1,y_2 >0$.
        We define $y_0 := \min(y_1,y_2)$. Let $u \in \mathbb{S}^{d-1}$ and $y \in (0,y_0)$. To prove that $g^{-1}_{1,u}(y) \geq A$, we need to prove that for every $r \leq A$, $g_u(r) > y$. If $r\leq r_0$, then $g_u(r) > y_1 \geq y$. If $r \in [r_0,A]$, then \[g_u(r) \geq \frac{1}{2}f_u(r) \geq \frac{1}{2}f_u(A) \geq y_2 \geq y_0 > y.\]
        We thus have $g^{-1}_{1,u}(y) \geq A$. In conclusion, we proved \eqref{eq:toprove_conv_infty_uniform} for $i=1.$
  The result for $i=2$ follows by observing that $g_{2,u}^{-1} \geq g_{1,u}^{-1}$.
\end{proof}
    
    \begin{prop}\label{prop:asymptotic_generalized_inverse_uniform}
    The functions $g_{i,u}^{-1}$ for $i \in \{1,2\}$ have the same asymptotics as $f_u^{-1}$ in Proposition~\ref{prop:asympotics_inverse_equivalent_uniform}, that is, with $p=\frac{2}{d-1}$:
    \begin{equation*}
        g_{i,u}^{-1}(y) = a_u \left[ p \log \left( \frac{1}{y} \right) - \log \log \left( \frac{1}{y} \right) + \log\left( \frac{b_u}{p} \right) \right] + o(1)
    \end{equation*}
    uniformly in $u$.
    As a consequence,  uniformly in $u$,
    \begin{align*}
        r_{N,u} &= \frac{1}{\gamma_u} \log N - \frac{d-1}{2\gamma_u} \log \log N + \frac{1}{\gamma_u} \log \left( \frac{\gamma_u ^{\frac{d-1}{2}} C_u}{\alpha} \right)+ o(1),\\
        R_{N,u} &= \frac{1}{\gamma_u} \log N - \frac{d-1}{2\gamma_u} \log \log N + \frac{1}{\gamma_u} \log \left( \frac{\gamma_u ^{\frac{d-1}{2}} C_u}{\beta} \right)+ o(1).
    \end{align*}
    \end{prop}
    
    \begin{proof}
        Let $y >0$ and $x = g_{i,u}^{-1}(y)$. By \eqref{eq:composition_generalized_inverse}, we have $g_{u}(x) = y$. Besides, since $g_{u} \underset{+\infty}{\sim} f_u$, uniformly in $u$, there exists a function $q_u$ that tends to $1$ at $+\infty$, uniformly in $u$, such that $g_u=q_uf_u$. Therefore
        \begin{equation*}
            g_{i,u}^{-1}(y)= x
                =f_u^{-1}(f_u(x))
                        = f_u^{-1}\left( \frac{g_{u}(x)}{q_u(x)} \right)= f_u^{-1}\left( \frac{y}{q_u\bigl( g_{i,u}^{-1}(y) \bigr)} \right).
        \end{equation*}
        We denote $\widetilde{q}_{i,u}(y) = q_{u}\left( g_{i,u}^{-1}(y) \right)$. Since $g_{i,u}^{-1}(y) \xrightarrow[y \to 0]{}  +\infty$ uniformly in $u$ according to Lemma~\ref{lem:conv_infty_uniform}, and $q_u(r) \xrightarrow[r \to +\infty]{}1$ uniformly in $u$, we have $\widetilde{q}_{i,u}(y) \xrightarrow[y \to 0]{}1$ uniformly in $u$. Therefore, Proposition~\ref{prop:asympotics_inverse_equivalent_uniform} leads to the existence of a function $\varepsilon$ that tends to $0$ and does not depend on $u$ such that
        \begin{align}
            g_{i,u}^{-1}(y)   &= f_u^{-1}\left( \frac{y} {\widetilde{q}_{i,u}(y)} \right)\nonumber\\
                        &= a_u \left[ p \log \left( y^{-1} \widetilde{q}_{i,u}(y) \right) - \log \log \left( y^{-1} \widetilde{q}_{i,u}(y) \right) + \log \left( \frac{b_u}{p} \right) \right] + \varepsilon\left( \frac{y}{\widetilde{q}_{i,u}(y)} \right)\nonumber\\
                        &= a_u \left[ p \log\left( y^{-1} \right) + p \log \left( \widetilde{q}_{i,u}(y) \right) - \log \left( \log \left( y^{-1}\right) + \log \left( \widetilde{q}_{i,u}(y) \right) \right) + \log\left( \frac{b_u}{p} \right) \right] + \varepsilon\left( \frac{y}{\widetilde{q}_{i,u}(y)} \right) \nonumber\\
                        &= a_u \left[ p \log\left( y^{-1} \right) + p \log \left( \widetilde{q}_{i,u}(y) \right) - \log \left( \log \left( y^{-1}\right) \left( 1 + \frac{\log \left( \widetilde{q}_{i,u}(y) \right)}{\log \left( y^{-1}\right)} \right) \right) + \log\left( \frac{b_u}{p} \right) \right] + \varepsilon\left( \frac{y}{\widetilde{q}_{i,u}(y)} \right)\nonumber\\
                        &= a_u \left[ p \log\left( y^{-1} \right) - \log  \log \left( y^{-1}  \right) + \log\left( \frac{b_u}{p} \right) \right] \nonumber\\&\qquad\qquad\qquad\qquad\qquad\qquad\qquad+ a_u p \log \left( \widetilde{q}_{i,u}(y) \right) - a_u \log\left( 1 + \frac{\log \left( \widetilde{q}_{i,u}(y) \right)}{\log \left( y^{-1}\right)} \right) + \varepsilon\left( \frac{y}{\widetilde{q}_{i,u}(y)} \right).\label{eq:three_last_terms}
        \end{align}
        We shall prove that the three terms appearing in \eqref{eq:three_last_terms} tend to $0$ uniformly in $u$ as $y \to 0$.
        It is clear for $a_u p \log \left( \widetilde{q}_{i,u}(y) \right)$, because $a_u$ is bounded (by a continuity plus compactness argument) and $\widetilde{q}_{i,u}(y)$ tends to $1$ uniformly in $u$. Similarly, $\frac{\log \left( \widetilde{q}_{i,u}(y) \right)}{\log \left( y^{-1}\right)}$ tends to $0$ uniformly in $u$ as $y \to 0$, so the same conclusion holds for the second term in \eqref{eq:three_last_terms}. Finally, $\frac{y}{\widetilde{q}_{i,u}(y)}$ tends to $0$ uniformly in $u$ as $y \to 0$. Therefore, $\varepsilon \bigl( \frac{y}{\widetilde{q}_{i,u}(y)} \bigr)$ tends to $0$ uniformly in $u$.
    \end{proof}
    
    We can summarize it into the following theorem.
    
\begin{thm}\label{thm:limit_shape_d_fixed}
    We have \[ \lim_{N \to +\infty}\frac{r_{N,u}}{\log N} = \lim_{N \to +\infty} \frac{R_{N,u}}{\log N} = \frac{1}{\gamma_u} = \frac{1}{\Gamma^{-1}(u) \cdot u}\] and these limits are uniform in $u$.
    This means that the limit shape of the LASM rescaled by $\log N$ is delimited by the curve 
    \begin{equation}\label{eq:limit_shape_in_spheric_coordinates}
        \mathcal C = \left\{ \frac{u}{\Gamma^{-1}(u) \cdot u}~|~u \in \mathbb{S}^{d-1}  \right\},
    \end{equation}
    and this limit is uniform in the sense of Definition~\ref{def:cv_sets_limit_shape}.
    
    Besides, since the asymptotics of $r_{N,u}$ and $R_{N,u}$ are the same until the constant term, there exists a constant $\tau$ such that \[ \sup_{N \in \N, u \in \mathbb{S}^{d-1}} \left(R_{N,u} - r_{N,u}\right) \leq \tau, \]
    hence \[ \sup_{ u \in \mathbb{S}^{d-1}} \frac{1}{\log N} \left(R_{N,u} - r_{N,u}\right) \leq \frac{\tau}{\log N}. \]
\end{thm}

In the next result we characterise the limit shape \eqref{eq:limit_shape_in_spheric_coordinates} in terms of duality (see Appendix~\ref{app:convexity} for various reminders concerning duality in our context). 

\begin{prop}
\label{prop:dual_curves}
The limit shape of the sandpile defined by \eqref{eq:limit_shape_in_spheric_coordinates} and the set $\partial \mathcal{T} := \rho^{-1}(\{1\})$ introduced in Definition~\ref{def:rho_T} are dual curves, in the sense of Definition~\ref{def:dual_curve}.
\end{prop}

\begin{proof}
    Let $t_0 \in \partial \mathcal{T}$. The normalized gradient of $\rho$ at $t_0$ is, by definition, $\Gamma(t_0)$, so the normal cone to $\mathcal{T}$ at $t_0$ consists of the non-negative multiples of $\Gamma(t_0)$. By Definition~\ref{def:dual_curve}, the point $t_0^*$ associated with $t_0$ in the dual curve of $\partial \mathcal{T}$ is $t_0^* = \frac{\Gamma(t_0)}{\Gamma(t_0) \cdot t_0}$.
    Writing $t_0^* = r u$ in spherical coordinates, we get $u = \Gamma(t_0)$ and $r = \frac{1}{\Gamma(t_0) \cdot t_0} = \frac{1}{u \cdot \Gamma^{-1}(u)}$, hence $t_0^* = \frac{1}{u \cdot \Gamma^{-1}(u)} u$ is in the limit shape defined in \eqref{eq:limit_shape_in_spheric_coordinates}.
    
    Conversely, if $\frac{1}{\Gamma^{-1}(u) \cdot u}u$ is in the limit shape defined by \eqref{eq:limit_shape_in_spheric_coordinates}, then setting $t_0 := \Gamma^{-1}(u)$, the previous computations show that $\frac{1}{\Gamma^{-1}(u) \cdot u} u = t_0^*$ is in the dual of $\partial \mathcal{T}$.
\end{proof}

Since duality preserves convexity (see Proposition~\ref{prop:duality_preserves_convexity}), we obtain the following.
\begin{prop}
The limit curve $\mathcal{C}$ in \eqref{eq:limit_shape_in_spheric_coordinates} is the boundary of a convex set.
\end{prop}

\section{Behavior of the limit curve when the leakiness parameter goes to zero or infinity}

We study the limit curve $\mathcal C$ given in \eqref{eq:limit_shape_in_spheric_coordinates} in two special regimes, first when the leakiness parameter goes to infinity, and second when it goes to zero (the massless case).
In both cases, the general strategy is to work on the level set $\partial \mathcal{T}$, which is the dual curve of the limit shape of the sandpile, according to Proposition~\ref{prop:dual_curves}, and to use duality to get the limit in the regimes studied. Useful properties of dual curves, especially Proposition~\ref{prop:convergence_duality} which links the convergence of convex sets with the convergence of their duals, are recalled in Appendix~\ref{app:convexity}.

\subsection{Infinite leakiness parameter case}
\label{subsec:infinite_mass_case}

Throughout Section~\ref{subsec:infinite_mass_case}, in addition to the running Assumption~\ref{assump:main_assumptions}, we will assume that the supports of the random walks are all bounded.

Let us first consider the uncolored case, i.e., $p=1$. In this simpler case, the limit shape of the sandpile as $\lp$ goes to infinity is easily understood using the support of the random walk.  The Laplace transform of the associated KRW is denoted by $L\mu(t)=\sum_{x\in\mathbb Z^d}e^{t\cdot x}\mu(x)$ ($\mu$ is $\mu_{1,1}$ in our notation~\eqref{eq:def_mu_i,j}). Let $\mathcal C$ be the curve that defines the boundary shape of the LASM, see \eqref{eq:limit_shape_in_spheric_coordinates}.

\begin{prop}\label{prop:cv_level_set_p=1}
We denote $\widetilde{\mu} = \lp \mu$, which is a probability measure and does not depend on $\lp$, according to notation~\eqref{eq:def_mu_i,j}.
The rescaled level set $\frac{1}{\log \lp} \left\{ t \in \R^d~|~L\mu(t) \leq 1 \right\}$ is between the polytopes
\begin{equation}\label{eq:limit_polytope}
    \bigcap_{x \in \supp{\mu}} \left\{ s \in \R^d~|~s \cdot x \leq 1 \right\}
\end{equation} 
and 
\[ \bigcap_{x \in \supp{\mu}} \left\{ s\in \R^d~|~s \cdot x \leq 1 - \frac{\log \widetilde{\mu}(x)}{ \log \lp} \right\}.\]
As a consequence, if the support of $\mu$ is finite, the rescaled level set converges uniformly to the polytope \eqref{eq:limit_polytope} in the sense of Definition~\ref{def:cv_sets_limit_shape} (and therefore also for the Hausdorff distance).
In that case, the limit polytope \eqref{eq:limit_polytope} is the dual of the convex hull of the support of $\mu$.
\end{prop}

See Figure~\ref{fig:step-set_example} for an illustration of Proposition~\ref{prop:cv_level_set_p=1}. Using general results on duality, the previous result implies that, when $p=1$, the curve $(\log \lp)\cdot \mathcal{C}$ converges to the convex hull of the support of the random walk as $m$ goes to infinity. Moreover, 
Proposition~\ref{prop:cv_level_set_p=1}  can be interpreted using first passage times, as will be done in Section~\ref{subsec:first_passage}.

\begin{proof}
Let $t \in \R^d$. Writing $s = \frac{t}{\log\lp}$, we have
\[L\mu(t) = \sum_{x \in \supp \mu} e^{(\log\lp) (s \cdot x -1)}\widetilde{\mu}(x).\]
If $s \cdot x \leq 1$ for every $x\in\supp\mu$, then $L\mu(t) \leq \sum_{x \in \supp \mu} \widetilde{\mu}(x) = 1$.
Conversely, if $L\mu(t) \leq 1$, then each term of the sum is no bigger than $1$, that is $s \cdot x \leq 1- \frac{\log \widetilde{\mu}(x)}{\log\lp}$.

Using the notations of Definition~\ref{def:cv_sets_limit_shape}, writing $s = r u$, we can define
\begin{equation*}
r_{N,u} = \frac{1}{u \cdot x}\quad\text{and}\quad
R_{N,u} = \frac{1}{u \cdot x} + \frac{1}{\log\lp} \max_{\substack{y \in \supp \mu \\ y \cdot u >0}} \frac{ - \log \widetilde{\mu}(y)}{ u \cdot y },
\end{equation*}
where $x$ is the element of the support that maximizes $u \cdot x$. It is then clear that $r_{N,u}$ and $R_{N,u}$ both converge uniformly to $\frac{1}{u \cdot x}$.

The fact that \eqref{eq:limit_polytope} is the dual of the convex hull of $\supp \mu$ follows directly from Proposition~\ref{prop:dual_polytope}.
%%To prove that \eqref{eq:limit_polytope} is the dual of the convex hull of $\supp \mu$, we first observe that \[\bigcap_{x \in \supp{\mu}} \left\{ s \in \R^d~|~s \cdot x \leq 1 \right\} = \bigcap_{x \in \mathrm{extr}(\supp{\mu})} \left\{ s \in \R^d~|~s \cdot x \leq 1 \right\}\] where $\mathrm{extr}(K)$ is the set of extreme points of $K$.
%%Thus the faces of the polytope \eqref{eq:limit_polytope} are the hyperplanes 
%%\begin{equation}\label{eq:faces_limit_polytope}
%%\left\{ s \in \R^d~|~s \cdot x = 1 \right\}
%%\end{equation}
%%for $x \in \mathrm{extr}(\supp \mu)$. These hyperplanes are their own tangent spaces, with normalised equation $s\cdot x = 1$, hence the dual point associated with the face \eqref{eq:faces_limit_polytope} is $x$.
\end{proof}

Let us now consider the general case in $p$. Our main objective is to state and prove Theorem~\ref{thm:limit_shape_as_convex_hull_of_returning_walks}. We first introduce useful notation and preliminary results. For $q\in\{1,\ldots,p\}$, define the polytope $\mathcal Q_\pm^{(q)}$ as follows. First, for $q=1$, we let
\begin{equation*}
    \mathcal Q_\pm^{(1)} = \bigcap_{1\leq i\leq p} \bigcap_{x \in \supp{\mu_{i,i}}} \left\{ s\in \R^d~|~s \cdot x \leq 1 \pm c^{(i)}_\lp \right\},
\end{equation*}
where $c^{(i)}_\lp$ are some constants.
For $q\geq 2$, we define
\begin{equation*}
    \mathcal Q_\pm^{(q)} = \bigcap_{i_1,i_2,\ldots,i_q}\bigcap_{\substack{x_{1,2}\in\supp{\mu_{i_1,i_2}}\\\cdots\\x_{q,1}\in\supp{\mu_{i_q,i_1}}}} \left\{ s\in \R^d~|~s \cdot (x_{1,2}+x_{2,3}+\cdots +x_{q,1}) \leq q\pm c^{(i_1,\ldots,i_q)}_\lp \right\},
\end{equation*}
where $c^{(i_1,\ldots,i_q)}_\lp$ are some constants. Then we introduce 
\begin{equation}
\label{eq:general_def_P_pm}
    \mathcal P_\pm = \bigcap_{1\leq q\leq p} \mathcal Q_\pm^{(q)}.
\end{equation}
The above polytopes obviously satisfy $\mathcal P_-\subset \mathcal P_+$. They both depend on a collection of constants
\begin{equation*}
    \bigl\{c_\lp^{(i_1,\ldots,i_q)}~\vert~1\leq i_1,\ldots,i_q\leq p,\ 1\leq q\leq p\bigr\}.
\end{equation*}
 For the sake of brevity, we will name these constants $c_\lp$. When they are all equal to zero, we have $\mathcal{P}_- = \mathcal{P}_+$, which we simply denote $\mathcal{P}$.

\begin{prop}
\label{prop:level_set_d_to_infinity}
Let $p\geq 1$. We assume that $\lp_1=\cdots = \lp_p=\lp\to \infty$. There exist constants $c_\lp$ going to $0$ such that the rescaled level set $\frac{1}{\log\lp}\{t\in\mathbb R^{d}~\vert~\rho(t)\leq 1\}$ lies between the two polytopes $\mathcal P_-$ and $\mathcal P_+$, as defined in \eqref{eq:general_def_P_pm}. In particular, it converges to the polytope $\mathcal{P}$.
\end{prop}

\begin{proof}
Performing the change of variable $s=\frac{t}{\log\lp}$, we first prove that the level set $\frac{1}{\log\lp}\{t\in\mathbb R^{d}~\vert~\rho(t)\leq 1\}$ is included in some $\mathcal P_+$, for suitable constants $c_\lp$ going to zero when $\lp\to\infty$.

If $\rho(t)\leq 1$ then for all $q\in\{1,\ldots,p\}$, one has $\Tr \bigl( (L \mu)^q\bigr)\leq p$. For $q=1$, we obtain that 
\begin{equation*}
    \sum_{1\leq i\leq p} \sum_{x\in\supp \mu_{i,i}} e^{\log\lp (s\cdot x-1)}\widetilde \mu_{i,i}(x) \leq p,
\end{equation*}
using our notation $\widetilde{\mu}_{i,i}=\lp\mu_{i,i}$ ($\widetilde{\mu}_{i,i}$ is a sub-probability measure). Accordingly, for all $1\leq i\leq p$ and all $x\in\supp \mu_{i,i}$, we should have 
\begin{equation*}
    s\cdot x\leq 1 + \frac{\log\frac{p}{\widetilde\mu_{i,i}(x)}}{\log\lp}.
\end{equation*}
This immediately gives that $s$ should belong to $\mathcal Q_+^{(1)}$, with some explicit constants 
\begin{equation*}
   c_\lp^{(i)} =  \sup_{x\in\supp \mu_{i,i}}  \frac{\log\frac{p}{\widetilde\mu_{i,i}(x)}}{\log\lp},
\end{equation*}
which obviously go to zero when $\lp\to\infty$.

We now look at the condition $\Tr \bigl( (L \mu)^q\bigr)\leq p$ for some $q\geq 2$. Using the definition of the trace, we get
\begin{equation}
\label{eq:expansion_trace}
    \Tr \bigl( (L\mu)^q\bigr) =\sum_{i_1,i_2,\ldots,i_q} \sum_{\substack{x_{1,2}\in\supp{\mu_{i_1,i_2}}\\\cdots\\x_{q,1}\in\supp{\mu_{i_q,i_1}}}} e^{\log\lp (s\cdot (x_{1,2}+x_{2,3}+\cdots+x_{q,1})-q)}\widetilde \mu_{i_1,i_2}(x_{1,2})\cdots \widetilde \mu_{i_q,i_1}(x_{q,1}) \leq p.
\end{equation}
By the same argument as above, we deduce that $s\in \mathcal Q_+^{(q)}$, with the constants
\begin{equation*}
    c_\lp^{(i_1,\ldots,i_q)} =  \sup_{\substack{x_{1,2}\in\supp{\mu_{i_1,i_2}}\\\cdots\\x_{q,1}\in\supp{\mu_{i_q,i_1}}}} \frac{\log\frac{p}{\widetilde \mu_{i_1,i_2}(x_{1,2})\cdots \widetilde \mu_{i_q,i_1}(x_{q,1})}}{\log\lp}.
\end{equation*}
We immediately conclude that $s\in\mathcal P_+$, with constants $c_\lp$ all going to $0$ (due to the boundedness of the supports).

We now assume that $s\in \mathcal P_-$ for some constants $c_\lp$ going to zero and prove that $\rho(t)\leq 1$. Using \eqref{eq:expansion_trace} and the definition of $\mathcal P_-$, we obtain
\begin{equation*}
    \Tr \bigl( (L \mu)^q\bigr) \leq \sum_{i_1,i_2,\ldots,i_q} \sum_{\substack{x_{1,2}\in\supp{\mu_{i_1,i_2}}\\\cdots\\x_{q,1}\in\supp{\mu_{i_q,i_1}}}} e^{-(\log\lp) c_\lp^{(i_1,\ldots,i_q)}}\widetilde \mu_{i_1,i_2}(x_{1,2})\cdots \widetilde \mu_{i_q,i_1}(x_{q,1})\leq \sum_{i_1,i_2,\ldots,i_q}e^{-(\log\lp) c_\lp^{(i_1,\ldots,i_q)}}.
\end{equation*}
We see that adjusting the constants $c_\lp$ (for instance $\frac{1}{\sqrt{\log\lp}}$), all the traces $\Tr \bigl( (L \mu)^q\bigr)$ can be made as small as wanted. Using then Lemma~\ref{lem:control_spectral_gap}, one can also control the spectrum $\rho(t)$ and make it less than one. The proof is complete. 
\end{proof}

\begin{lemma}
\label{lem:control_spectral_gap}
    One has 
    \[\sup\left\{ \rho(M)~|~ M \in \mathcal{M}_p(\R),~ \left|\Tr(M)\right| \leq \alpha,\ldots,  \left|\Tr\left(M^p\right)\right| \leq \alpha \right\} \xrightarrow[\alpha\to 0 ]{}0.\]
\end{lemma}

\begin{proof}
Let us write $a_i(M)$ for the coefficients of the characteristic polynomial of $M$, so that 
\begin{equation*}
    \chi_M = X^p + a_{p-1}(M) X^{p-1} + \cdots + a_1(M) X + a_0(M).
\end{equation*}
By continuity of the roots of monic polynomials with fixed degree, if we write $\rho(P)$ for the largest root of $P$ in modulus, the function \[(a_{p-1}, \ldots, a_1, a_0) \in \R^p \mapsto \rho\left( X^p + a_{p-1}X^{p-1} + \ldots + a_1 X +a_0 \right)\] is continuous at $(0, \ldots,0)$, where it takes the value $0$. Therefore, it is sufficient to prove that for all $i \in \{0, \ldots, p-1\}$,
\begin{equation}\label{eq:condition_suffisante_coefficients}
\sup\left\{ \left| a_i(M) \right|~|~ M \in \mathcal{M}_p(\R),~ \left|\Tr(M) \right|\leq \alpha, \ldots,  \left|\Tr\left(M^p \right) \right|\leq \alpha \right\} \xrightarrow[\alpha \to 0]{}0.
\end{equation}
To do so, we use the Newton identities to write the elementary symmetric polynomials in terms of Newton sums. With Vieta's formulas, it leads to an expression of the non-leading coefficients of a monic polynomial $P$ of degree $p$ in terms of the $\sum_{i=1}^p \lambda_i^k$ for $k \in \{1, \ldots, p\}$, where the $\lambda_i$ are the roots of $P$. Applying this to the characteristic polynomial of a matrix $M$, it means that all the $a_i(M)$ have a polynomial expression in terms of the $\Tr\left(M^k\right)$ for $k \in \{1, \ldots, p\}$, and it is easy to see that this polynomial expression has no constant term. Therefore, the convergences mentioned in \eqref{eq:condition_suffisante_coefficients} become obvious.
\end{proof}

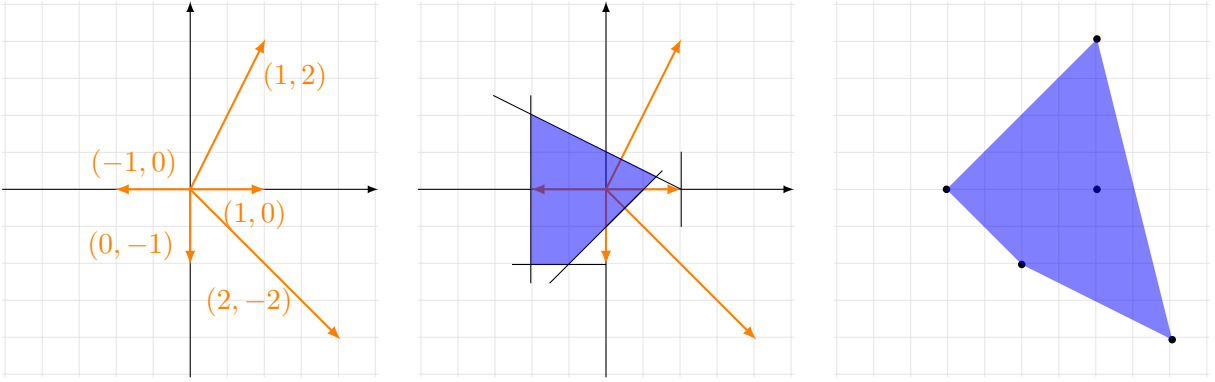
\begin{figure}
\begin{center}
\begin{tikzpicture}[line/.style={>=latex}] 
\draw[step=14pt, color=black!10] (-2.5, -2.5) grid (2.5, 2.5);
\draw[->, line] (-2.5, 0) -- node [below, very near end] {} (2.5, 0);
\draw[->, line] (0, -2.5) -- node [right, very near start] {} (0, 2.5);
\draw[->, line, color=orange, thick] (0, 0) -- node [right=2pt, near end]    {$(1,2)$} (1,2);
\draw[->, line, color=orange, thick] (0, 0) -- node [right=10pt, below]    {$(1,0)$} (1,0);
\draw[->, line, color=orange, thick] (0, 0) -- node [left, near end]    {$(2,-2)$} (2,-2);
\draw[->, line, color=orange, thick] (0, 0) -- node [left=2pt, near end]    {$(0,-1)$} (0,-1);
\draw[->, line, color=orange, thick] (0, 0) -- node [left, above, near end]    {$(-1,0)$} (-1,0);
\end{tikzpicture}
\quad 
\begin{tikzpicture}[line/.style={>=latex}] 
\draw[step=14pt, color=black!10] (-2.5, -2.5) grid (2.5, 2.5);
\draw[->, line] (-2.5, 0) -- node [below, very near end] {} (2.5, 0);
\draw[->, line] (0, -2.5) -- node [right, very near start] {} (0, 2.5);
\draw[->, line, color=orange, thick] (0, 0) -- node [right=2pt, near end]    {} (1,2);
\draw[->, line, color=orange, thick] (0, 0) -- node [right=10pt, below]    {} (1,0);
\draw[->, line, color=orange, thick] (0, 0) -- node [left, near end]    {} (2,-2);
\draw[->, line, color=orange, thick] (0, 0) -- node [left=2pt, near end]{} (0,-1);
\draw[->, line, color=orange, thick] (0, 0) -- node [left, above, near end]    {} (-1,0);
\draw[-, line] (1,0) -- node [right, very near start] {} (-1.5,1.25);
\draw[-, line] (-1,1.25) -- node [right, very near start] {} (-1,-1.25);
\draw[-, line] (-1.25,-1) -- node [right, very near start] {} (0,-1);
\draw[-, line] (0.5,0) -- node [right, very near start] {} (0,-0.5);
\draw[-, line] (0.75,0.25) -- node [right, very near start] {} (-0.75,-1.25);
\draw[-, line] (1,0.5) -- node [right, very near start] {} (1,-0.5);
\fill [opacity=0.5,blue]
  (-0.5,-1) -- (-1,-1) -- (-1,1)  -- (0.675,0.175) -- cycle;
\end{tikzpicture}
\quad
\begin{tikzpicture}[line/.style={>=latex}] 
\draw[step=14pt, color=black!10] (-2.5, -2.5) grid (2.5, 2.5);
\fill (1,2) circle [radius=.05];
\fill (1,0) circle [radius=.05];
\fill (2,-2) circle [radius=.05];
\fill (0,-1) circle [radius=.05];
\fill (-1,0) circle [radius=.05];
\fill [opacity=0.5,blue]
  (1,2) -- (2,-2) -- (0,-1) -- (-1,0) -- cycle;
\end{tikzpicture}
\end{center}
\caption{Left picture: the step set $\{(1,0),(2,-2),(0,-1),(-1,0),(1,2)\}$. Middle picture: the limit level set. Right picture: the Newton polytope associated with the step set}
\label{fig:step-set_example}
\end{figure}

We now use general results on duality to prove that the rescaled curve $\mathcal{C}$ converges to the dual polytope of $\mathcal{P}$. In order to apply Appendix~\ref{app:convexity}, we need $\mathcal{P}$ to be compact, hence the following lemma.
\begin{lemma}
The limit polytope $\mathcal{P}$ is bounded.
\end{lemma}
\begin{proof}
The polytope $\mathcal{P}$ is convex. Therefore, to prove the lemma, we only need to show that $\mathcal{P}$ does not contain a half-line starting from $0$.

Let $u \in \mathbb S^{d-1}$.
We write $H(q)$ for the following assertion: there exist $i_1, \ldots, i_q \in \{1, \ldots, p\}$ and $x_{1,2} \in \supp \mu_{i_1,i_2}, \ldots, x_{q,1} \in \supp \mu_{i_q, i_1}$ such that $u \cdot \left(x_{1,2}+ \cdots + x_{q,1}\right) >0$. We want to prove that $H(q)$ is true for some $q \in \{1, \ldots, p\}$, so that the condition $s\cdot \left(x_{1,2}+ \cdots + x_{q,1}\right) \leq q$ prevents $\mathcal{P}$ from containing the half-line $[0,+\infty) u$. 

By the assumption of irreducibility, $H(q)$ is true for some $q \in \N$. We choose the minimal $q$. Assume by contradiction that $q > p$. Then, in the path $i_1 \to \cdots \to i_q \to i_1$, there is a cycle $i_{k} \to \cdots \to i_{\ell} = i_k$ with $1 \leq \ell-k \leq p < q$. Since $H(\ell-k)$ is false, $u \cdot (x_{k, k+1} + \cdots + x_{\ell-1,\ell}) \leq 0$. Therefore, we have 
\begin{align*}
u \cdot \left(x_{1,2} + \cdots + x_{k-1,k} + x_{\ell, \ell+1} + \cdots + x_{q,1} \right) &= u \cdot \left(x_{1,2} + \cdots + x_{q,1} \right) - u \cdot (x_{k, k+1} + \cdots + x_{\ell-1,\ell})\\
&\geq u \cdot \left(x_{1,2} + \cdots + x_{q,1} \right) >0,
\end{align*}
where \[x_{1,2} \in \supp \mu_{i_1, i_2}, \ldots, x_{k-1,k} \in \supp \mu_{i_{k-1}, i_k}, x_{\ell, \ell+1} \in \supp \mu_{i_\ell, i_{\ell+1}} = \supp \mu_{i_k, i_{\ell+1}}, \ldots, x_{q,1} \in \supp \mu_{i_q,i_1}.\]
This proves $H(q-(\ell-k))$, which contradicts the minimality of $q$.
\end{proof}

\begin{prop}\label{prop:polytopeBase}
The limit polytope $\mathcal{P}$ is the dual of the convex hull of the set 
\[ \mathcal{X} = \bigcup_{q\in \{1, \ldots, p\}} \bigcup_{i_1, \ldots, i_q \in \{1, \ldots, p\}} \left\{ \frac{x_{1,2} + \cdots + x_{q,1}}{q}~|~x_{1,2} \in \supp \mu_{i_1,i_2}, \ldots, x_{q,1} \in \supp \mu_{i_q, i_1} \right\} .\]
\end{prop}

\begin{proof}
It is exactly Proposition~\ref{prop:dual_polytope}.
\end{proof}

\begin{thm} \label{thm:limit_shape_as_convex_hull_of_returning_walks}
    If the support of $\mu$ is finite, the rescaled shape of the sandpile, $(\log \lp)\cdot \mathcal{C}$, where $\mathcal{C}$ is defined in \eqref{eq:limit_shape_in_spheric_coordinates}, converges to the convex hull of $\mathcal{X}$ as $\lp$ goes to $+\infty$. The convergence is uniform in the sense of Definition~\ref{def:cv_sets_limit_shape} and consequently holds for the Hausdorff distance.
    
    In particular, in the uncolored case $p=1$ the limit shape of the sandpile when $\lp\to\infty$ is the convex hull of the support of the measure $\mu$. 
\end{thm}
\begin{proof}
Using Propositions~\ref{prop:dual_curves}~and~\ref{prop:dual_homotheties_rotations}, $(\log\lp)\cdot \mathcal{C}$ is the dual curve of $\frac{1}{\log \lp}\left\{ \rho(t)\leq 1  \right\}$. The inclusions of Proposition~\ref{prop:level_set_d_to_infinity} imply that 
\[\left(1 - \varepsilon_\lp\right) \mathcal{P} \subset \frac{1}{\log\lp}\left\{ \rho(t)\leq 1  \right\} \subset \left(1 + \varepsilon_\lp\right) \mathcal{P}, \]
where $\varepsilon_\lp = \max \frac{c_\lp^{(i_1, \ldots, i_q)}}{q} \xrightarrow[\lp \to + \infty]{}0$. Therefore, according to Proposition~\ref{prop:convergence_duality}, $(\log\lp)\cdot \mathcal{C}$ converges to $\mathcal{P}^*$, which is the convex hull of $\mathcal{X}$.
\end{proof}

\subsection{First passage time interpretation of the limiting polytope in the infinite leakiness parameter region}
\label{subsec:first_passage}

Our main result in this section is the following:
\begin{thm}\label{thm:limit_shape_as_first_passage_front}
If the support of $\mu$ is finite, as $\lp$ goes to $+\infty$ the rescaled shape of the sandpile converges to the region of points that can be reached by $n$ steps of the random walk, rescaled by $n$.
\end{thm}
\begin{proof}
The theorem is immediate from Theorem~\ref{thm:limit_shape_as_convex_hull_of_returning_walks} and Proposition~\ref{prop:convexHull_vs_firstPassage} below.
\end{proof}

Let $\mathcal{X}$ be as in Proposition~\ref{prop:polytopeBase}. Given $(x,i)\in\mathbb{Z}^d\times\{1,\dots,p\}$, let $n_{x,i}$ be the smallest positive integer such that $(x,i)$ can be reached in $n_{x,i}$ steps by our killed random walk. Denote by $A_n$ the set of points for which this first passage time is at most $n$, i.e., 
\begin{equation*}
    A_n=\bigl\{(x,i)\in\mathbb{Z}^d\times\{1,\dots,p\}:n_{x,i}\leq n\bigr\}.
\end{equation*}

\begin{prop}
\label{prop:convexHull_vs_firstPassage}
Suppose the support of $\mu$ is finite. As $n$ converges to $\infty$, the set $A_n$, scaled by $n$, converges to the convex hull of $\mathcal{X}$ in the Hausdorff distance, i.e.
$$\lim_{n\to\infty}d_H\left(\conv(\mathcal{X}),\frac 1n\cdot A_n\right)=0.$$
\end{prop}
\begin{proof}
First, consider the case $p=1$. We will ignore the color index. Let $s=|\chi|$ and write $\mathcal{X}=\{x_1,x_2,\dots,x_s\}$ be the support of $\mu$, which is a finite subset of $\mathbb{Z}^d$ such that the $\mathbb{N}^d$-span of $\mathcal{X}$ is all of $\mathbb{Z}^d$. Without loss of generality we can assume $0\in \mathcal X$.

If $x\in A_n$, then there exist points $x_{i_1},\dots,x_{i_n}\in \mathcal{X}$ such that $x=\sum_{j=1}^nx_{i_j}$. Thus $$\frac 1n x=\sum_{j=1}^n\frac 1n x_{i_j}\in\conv(\mathcal{X}),$$ so $\frac 1n\cdot A_n\subset\conv(\mathcal{X})$.

To complete the argument we need to show that for any $y\in\conv(\mathcal{X})$, its distance from $\frac 1n\cdot A_n$ converges to $0$ uniformly in $y$ as $n\to\infty$. If $h$ is any fixed integer, then the Hausdorff distance between $\frac 1n\cdot A_n$ and $\frac{1}{n}\cdot A_{n+h}$ converges to $0$ as $n\to\infty$ since the support of $\mu$ is finite. Thus, it is enough to show that the distance of $y$ from $\frac 1{n}\cdot A_{n+h}$ converges to $0$ uniformly in $y$ as $n\to\infty$. For the constant $h$ we take the largest number of steps the random walk needs to reach a point in $s\cdot\conv(\mathcal{X})$, i.e. $h:=\max\{n_x:x\in s\cdot\conv(\mathcal{X})\cap\mathbb{Z}^d\}$. 

Given $y\in\conv(\mathcal{X})$, there is a point $x\in n\cdot\conv(\mathcal{X})\cap\mathbb{Z}^d$ such that $\Vert y-x/n\Vert\leq \frac 1n$. Since $x\in n\cdot\conv(\mathcal{X})\cap\mathbb{Z}^d$, there exist non-negative real numbers $\alpha_1,\dots,\alpha_s$, which sum to $1$, such that $x=\sum_{j=1}^s n\alpha_i x_i$. We can write
\begin{equation}
\label{eq:separateFractionalPart}
x=\sum_{i=1}^s\lfloor n\alpha_i\rfloor x_i+\sum_{i=1}^s\{ n\alpha_i\} x_i,
\end{equation}
where $\lfloor t\rfloor$ stands for the integer part of $t$ and $\{t\}$ for $t-\lfloor t\rfloor$. Since both $x$ and the first sum on the right-hand side are from $\mathbb{Z}^d$, so is the second sum. Since each coefficient $\{n\alpha_i\}\in[0,1)$, we have that $\sum_{i=1}^s\{ n\alpha_i\}\in[0,s)$. Thus $\sum_{i=1}^s\{ n\alpha_i\} x_i\in s\cdot\conv(\mathcal{X})$, so it can be written as the sum of at most $h$ elements of $\mathcal{X}$. Since $\sum_{i=1}^s\lfloor n\alpha_i\rfloor\leq n$, we see that $x$ can be written as the sum of at most $n+h$ elements from $\mathcal{X}$, hence $x\in A_{n+h}$. It follows that $d\left(y,\frac{1}{n}\cdot A_{n+h}\right)\leq \frac 1{n}$, completing the argument in the case $p=1$.

We now move to the case of a general $p$. First, let us introduce some notation. Given colors $i$ and $j$, let $k_{i,j}$ be the minimum number of steps the random walk requires to move from color $i$ to color $j$, and let $\bar{x}_{i,j}$ be the total displacement of one such minimum walk. Let $k=\max\{k_{i,1}:i\in\{1,\dots,p\}\}$. 

Now, suppose $(x,i)\in A_n\backslash A_{n-1}$. Thus, there are steps $x_{0,1},\dots,x_{{n-1},n}$, with $x_{\ell,{\ell+1}}\in\supp\mu_{i_\ell,i_{\ell+1}}$ for some colors $i_0=1,i_1,\dots,i_{n-1},i_n=i$, such that $x=\sum_{j=1}^nx_{j-1,j}$. By adding $k_{i,1}$ more steps $x_{n,n+1},\dots,x_{n+k_i-1,n+k_i}$ with $x_{n+\ell,n+\ell+1}\in\mu_{i_{n+\ell},i_{n+\ell+1}}$, $i_{n+k_i}=1$ and $$x_{n,n+1}+\dots+x_{n+k_i-1,n+k_i}=\bar{x}_{i,1},$$ we can ensure our walk starts and ends with the first color $1$. We will show that $x+\bar{x}_{i,1}\in (n+k)\cdot\conv(\mathcal{X})$ by induction on $n$. If no colors in $\{i_0,\dots,i_{n+k}\}$ repeat, we are done. If there are repeating colors, there exist $0\leq j_1<j_1+t\leq n+k$ such that the colors $i_{j_1},i_{j_1+1},\dots,i_{j_1+t}$ are all distinct except the first and the last one. Removing the portion $x_{i_{j_1},i_{j_1+1}},\dots,x_{i_{j_1+t-1},i_{j_1+t}}$ from the walk, we get a walk of length $n-t+k_i$ which respects colors and starts and ends with color $1$, so by induction, 
$$x+\bar{x}_{i,1}-\sum_{\ell=1}^t x_{i_1+\ell-1,i_1+\ell}\in (n+k-t)\cdot\conv(\mathcal{X}).$$ Since $\sum_{\ell=1}^t x_{i_1+\ell-1,i_1+\ell}\in t\cdot\conv(\mathcal{X})$, we get $x+\bar{x}_{i,1}\in (n+k)\cdot\conv(\mathcal{X})$ which implies
\begin{equation}
\label{eq:A_n_in_conv}
\frac{1}{n+k}(x+\bar{x}_{i,1})\in \conv(\mathcal{X}).
\end{equation}

Now, suppose $x\in n\cdot\conv(\mathcal{X})$. If we ignore colors, by the case $p=1$, there exists $h\in\mathbb{N}$ such that $x$ can be written as the sum of at most $n+h$ elements of $\mathcal{X}$. In the proof of the case $p=1$, in Equation~\eqref{eq:separateFractionalPart}, instead of separating the integer part, for each term of the form $\frac{x_{1,2}+\dots +x_{q,1}}q\in\mathcal{X}$, we can separate a multiple of $p!$, thus ensuring that the coefficient of $\frac{x_{1,2}+\dots +x_{q,1}}q\in\mathcal{X}$ will be a multiple of $q$. We can thus write $x$ as a non-negative integer linear combination of at most $n+h$ terms of the form $x_{1,2}+\dots+x_{q,1}$. Each of these terms corresponds to a walk that starts and ends at the same color. We can arrange this sum, so that all the walks that start with the same color are next to each other. For example we can start with all our walks that start and end in color $1$, followed by all the walks that start and end in color $2$, etc. This will not correspond to a color respecting random walk from $A_{n+h}$, but by adding at most $p-1$ walks of the type $\bar{x}_{i,j}$ between these, each of which has length at most $k$, we get $x+\sum \bar{x}_{i,j}\in A_{n+h+k(p-1)}$, so
$$\frac{1}{n}\left(x+\sum \bar{x}_{i,j}\right)\in \frac 1n\cdot A_{n+h+k(p-1)}.$$
Combining this with \eqref{eq:A_n_in_conv} gives the desired result since $h,k$ and $\bar{x}_{i,j}$ are independent of $x$ and $n$.
\end{proof}

\subsection{Zero leakiness parameter case}
\label{subsec:zero_mass}

In this section, we no longer assume that the supports of the random walks are  bounded.
We suppose that $\lp_1 = \cdots = \lp_p = \lp \to 1$. We denote $\widetilde{\mu}_{i,j} = \lp \mu_{i,j}$, which no longer depends on $\lp$, and $\widetilde \rho = \rho\left( L \widetilde\mu \right)$. With these notations, by Proposition~\ref{prop:dual_curves} the shape of the sandpile when $N \to +\infty$ for fixed $\lp$ is the dual of the set \[\left\{\widetilde \rho = \lp\right\}.\] We assume that $\nabla \widetilde\rho(0) = 0$. It means that the process with jumps $\widetilde{\mu}$ is centered (see \cite[Prop.~2.16]{Ba-24}). If the process is not centered, then the set $\left\{\widetilde \rho  = 1\right\}$ was studied in \cite{Ba-24} and does not require any scaling.

Under these assumptions, $\widetilde\rho$ is convex, has a strict, global minimum at $0$, and we remind \cite[Prop.~2.14]{Ba-24} the asymptotics of $\widetilde{\rho}$ at $0$:
\begin{equation}
\label{eq:asymptotics_rho}
    \widetilde{\rho}(t) = 1 + \frac{1}{2} t^\intercal \sigma t + o\left( \|t\|^2 \right),
\end{equation}
where $\sigma$ is the \emph{energy matrix} from \cite{Ba-24}. When $p=1$, the energy matrix is simply the covariance matrix of $\widetilde\mu$.

\begin{prop}
\label{prop:level_set_d_to_1}
Let $\delta \in (0,1)$. There exists $\lp_0>1$ such that for all $\lp \in (1,\lp_0)$, we have 
\[ \left\{ s \in \R^d~|~ \frac{1}{2}s^\intercal \sigma s \leq 1 - \delta \right\} \subset \frac{1}{\sqrt{\lp-1}}\left\{ t \in \R^d ~|~ \widetilde{\rho}(t) \leq  \lp  \right\} \subset \left\{ s \in \R^d~|~ \frac{1}{2}s^\intercal \sigma s \leq 1 + \delta \right\}.\]
As a consequence, the rescaled set $\frac{1}{\sqrt{\lp-1}}\left\{ \widetilde{\rho} \leq  \lp  \right\}$ converges uniformly to the ellipsoid $\left\{s \in \R^d~|~ \frac{1}{2} s^\intercal \sigma s \leq 1 \right\}$ as $\lp\to 1$.
\end{prop}
\begin{proof}
Since $\sigma$ is positive-definite by \cite[Prop.~2.8]{Ba-24}, $\alpha :=  \displaystyle\min_{u \in \mathbb{S}^{d-1}} u^\intercal \sigma u >0$. Let $\varepsilon \in \left(0,\frac{\alpha}{2} \right)$. The Taylor expansion \eqref{eq:asymptotics_rho} ensures that there exists a neighbourhood $V$ of $0$ such that for $t \in V$, 
\begin{equation}
\label{eq:control_little_o_epsilon}
    \left|\widetilde{\rho}(t) - 1 - \frac{1}{2} t^\intercal \sigma t\right| \leq \frac{\varepsilon}{2} \|t\|^2.
\end{equation}

Let $t \in \R^d$ such that $\widetilde\rho(t) \leq \lp$. Since $\widetilde \rho$ is convex and has a global strict minimum at $0$, if $\lp$ is close enough to $1$, then $t \in V$, which allows to use \eqref{eq:control_little_o_epsilon}. We set $s = \frac{t}{\sqrt{\lp-1}}$. Then 
\begin{equation}
\label{eq:outer_ellipsoid}
\frac{1}{2}s^\intercal \sigma s  = \frac{1}{2(\lp-1)} t^\intercal \sigma t \leq \frac{1}{\lp-1} \left( \widetilde\rho(t) - 1 + \frac{\varepsilon}{2} \|t\|^2 \right) \leq 1 + \frac{\varepsilon}{2} \|s\|^2.
\end{equation}
Writing $s = \|s\| u$ with $u \in \mathbb S^{d-1}$, \eqref{eq:outer_ellipsoid} becomes $\frac{1}{2} \|s\|^2 u^\intercal \sigma u \leq 1 + \frac{\varepsilon}{2}\|s\|^2$, so
\[\|s\|^2 \leq \frac{2}{u^\intercal \sigma u - \varepsilon} \leq \frac{4}{\alpha}, \]
and \eqref{eq:outer_ellipsoid} finally becomes 
\begin{equation*}
    \frac{1}{2}s^\intercal \sigma s \leq 1 + \frac{2}{\alpha} \varepsilon.
\end{equation*}

Conversely, let $\eta = \frac{1}{1 + \frac{\varepsilon}{\alpha}}$. Assume that $\frac{1}{2} s^\intercal \sigma s \leq \eta$. Then, writing $s = \|s\| u$ for $u \in \mathbb{S}^{d-1}$, we get $\|s\|^2 \leq \frac{2\eta}{u^\intercal \sigma u} \leq \frac{2\eta}{\alpha}$. Therefore, $\|t\| \leq \sqrt{\frac{2 \eta(\lp-1)}{\alpha}} \leq \sqrt{\frac{2(\lp-1)}{\alpha}}$. This ensures that if $\lp$ is close enough to $1$, we have $t \in V$, which allows us to use \eqref{eq:control_little_o_epsilon}. It yields
\begin{equation*}
\widetilde{\rho}(t) \leq 1 + \frac{1}{2}t^\intercal \sigma t + \frac{\varepsilon}{2} \|t\|^2 \leq 1 + (\lp-1) \eta + \varepsilon \frac{(\lp-1)\eta}{\alpha}= \lp.
\end{equation*}

To summarize, we have proved that for $\lp$ close enough to $1$, 
\[ \left\{ s \in \R^d~|~ \frac{1}{2}s^\intercal \sigma s \leq \frac{1}{1 + \frac{\varepsilon}{\alpha}} \right\} \subset \frac{1}{\sqrt{\lp-1}}\left\{ t \in \R^d ~|~ \widetilde{\rho}(t) \leq  \lp  \right\} \subset \left\{ s \in \R^d~|~ \frac{1}{2}s^\intercal \sigma s \leq 1 + \frac{2}{\alpha} \varepsilon \right\}. \]
Choosing $\varepsilon >0$ small enough so that $\frac{2}{\alpha} \varepsilon \leq \delta$ and $\frac{1}{1+ \frac{\varepsilon}{\alpha}} \geq 1 - \delta$, we obtain the announced inclusions.
\end{proof}

\begin{thm}\label{thm:limit_shape_ellipsoid}
    If $\nabla \widetilde{\rho}(0) = 0$, then the rescaled shape of the sandpile, $\sqrt{\lp-1}\cdot \mathcal{C}$, where $\mathcal{C}$ is defined in \eqref{eq:limit_shape_in_spheric_coordinates}, converges to the ellipsoid \[ \left\{ s \in \R^d ~|~ 2 s^\intercal \sigma^{-1} s \leq 1 \right\} \] as $\lp$ tends to $1$. The convergence is uniform in the sense of Definition~\ref{def:cv_sets_limit_shape} and consequently holds for the Hausdorff distance.
\end{thm}
\begin{proof}
Using Propositions~\ref{prop:dual_curves}~and~\ref{prop:dual_homotheties_rotations}, $\sqrt{\lp-1}\cdot \mathcal{C}$ is the dual curve of $\frac{1}{\sqrt{\lp-1}}\left\{ \widetilde{\rho} \leq  \lp  \right\}$. The inclusions of Proposition~\ref{prop:level_set_d_to_1} 
can be written as 
\[ \sqrt{1 - \delta}\left\{ s \in \R^d~|~ \frac{1}{2}s^\intercal \sigma s \leq 1  \right\} \subset \frac{1}{\sqrt{\lp-1}}\left\{ t \in \R^d ~|~ \widetilde{\rho}(t) \leq  \lp  \right\} \subset \sqrt{1 + \delta}\left\{ s \in \R^d~|~ \frac{1}{2}s^\intercal \sigma s \leq 1 \right\}.\]
Therefore, Proposition~\ref{prop:convergence_duality}, shows that $\sqrt{\lp-1}\cdot \mathcal{C}$ converges to the set \[\left\{ s \in \R^d ~|~ \frac{1}{2} s^\intercal \sigma s \leq 1 \right\}^* = \left\{ s \in \R^d ~|~ 2 s^\intercal \sigma^{-1} s \leq 1 \right\}, \] according to the computation of duals of ellipsoids in Proposition~\ref{prop:dual_ellipsoid}.
\end{proof}

\appendix{}
\section{Dual curves and dual convex sets}
\label{app:convexity}

In this appendix, we remind the definition of dual curves and give the properties that are used through the article. We denote by $K$ a compact, convex subset of $\R^d$ and $\mathcal{C} = \partial K$. Keep in mind that the cases we are interested in are the case where $K$ is a polytope and the smooth case, that is the case when $K$ has the form \[K = \left\{ x \in \R^d~|~f(x) \leq 1 \right\}\] for some smooth, convex function $f$ that takes some values strictly larger than $1$.

\begin{defi}
Let $x \in \mathcal C$. The normal cone to $K$ at $x$ is the set 
\begin{equation*}
    N_{K}(x) = \left\{ n \in \R^d~|~\forall y \in K,\ n \cdot (y-x) \leq 0 \right\}.
\end{equation*}
\end{defi}

\begin{prop}
In the smooth case, the normal cone at a point $x \in \mathcal C$ consists of the non-negative multiples of $\nabla f (x)$.
\end{prop}
\begin{proof}
We first prove that $\nabla f(x) \in N_K(x)$. Let $y \in K$. If $\nabla f(x) \cdot (y-x) >0$, then a Taylor expansion shows that for $\varepsilon >0$ small enough, $f(x + \varepsilon(y-x)) > f(x) = 1$, therefore $x + \varepsilon(y-x) \notin K$, which contradicts the convexity of $K$. Therefore, $\nabla f(x) \cdot (y-x) \leq 0$, i.e., $\nabla f(x) \in N_K(x)$, so its non-negative multiples also belong to the cone.

Conversely, let $n \in N_K(x) \setminus \{0\}$. By definition of the normal cone and of $K = f^{-1}((-\infty, 1])$, we have \[g(x) = \max_{f(y) = 1} g(y),\] where $g(y) = n \cdot y$. Besides, $\nabla f(x) \neq 0$ because $f$ is convex and we assumed that $1$ is not the minimum of $f$. Therefore, the Lagrange multiplier theorem ensures that $\nabla g(x) = n$ is a multiple of $\nabla f(x)$, which we write $n = \lambda \nabla f(x)$. It remains to prove that $\lambda > 0$. We know that $\nabla f(x) \in N_K(x)$, so if $\lambda \leq 0$, we would have $n \cdot (y-x) = \lambda \nabla f(x) \cdot (y-x) \geq 0$ and $n \cdot (y-x) \leq 0$ for all $y \in K$, so $K \subset x + n^\perp$. But we assumed that $f$ takes values strictly smaller than $1$, so $K$ has a non-empty interior and cannot be a subset of a hyperplane.
\end{proof}

\begin{lemma}
\label{lem:normal_cones}
\begin{enumerate}
    \item Normal cones are closed subsets of $\R^d$.
    \item For every $x \in \mathcal{C}$, $N_K(x) \neq \{0\}$.
    \item \label{it:lem:normal_cones_3} Let $u \in \R^d \setminus \{0\}$. There exists $x \in \mathcal{C}$ such that $u \in N_K(x)$.
    \item For every $t >0$ and $x \in \mathcal{C}$, $N_{tK}(tx) = N_{K}(x)$.
    \item For every $U \in O_d(\R)$ and $x \in \mathcal{C}$, $N_{UK}(Ux) = U N_{K}(x)$.
\end{enumerate}

\end{lemma}
\begin{proof}
Let $x \in \mathcal{C}$. Then \[N_K(x) = \bigcap_{y \in K} g_y^{-1}\left( (-\infty, 0] \right),\] where $g_y : n \in \R^d \mapsto n \cdot (y-x)$ is continuous. This shows that $N_K(x)$ is the intersection of closed subsets of $\R^d$, and thus it is a closed subset of $\R^d$.

We denote by $p_K$ the projection onto the closed convex set $K$, which is well defined and continuous according to the Hilbert projection theorem. The point $x$ belongs to $\mathcal{C} = \partial K$, so there exists a sequence $\left( z_k \right)_{k\geq 0}$ of $\R^d \setminus K$ that tends to $x$. Let $u_k = \frac{z_k - p_K(z_k)}{\left\|z_k - p_K(z_k)\right\|}$. The Hilbert projection theorem ensures that
\begin{equation}
\label{eq:normal_cone_not_singleton}
u_k \cdot (y-p_K(z_k)) \leq 0 \text{ for every $y \in K$}.
\end{equation}
By compactness of the unit sphere, there is a subsequence of $\left(u_k\right)_{k\geq 0}$ that converges to a unit vector $u$. Besides, $p_K(z_k)$ tends to $p_K(x) = x$ because $p_K$ is continuous. Therefore, taking the subsequential limit of \eqref{eq:normal_cone_not_singleton}, we get $u \cdot (y-x) \leq 0$ for all $y \in K$, hence $u \in N_K(x) \setminus \{0\}$.

Let $x \in K$ such that $u \cdot x = \max_{y \in K} u \cdot y$. Then $x \in \mathcal{C} = \partial K$, because if $x \in \mathring{K}$, then for $t >0$ small enough, $x+tu \in K$ and $u \cdot (x+tu) = u \cdot x + t > u\cdot x$, which contradicts the maximality of $u \cdot x$. By maximality of $u \cdot x$, we have $u \in N_K(x)$.

We have 
\begin{align*}
    n \in N_{tK}(tx) & \Longleftrightarrow \forall w \in tK,\quad n \cdot(w - tx)\leq 0\\
    & \Longleftrightarrow \forall y \in K,\quad n \cdot(ty - tx)\leq 0\\
    &\Longleftrightarrow \forall y \in K,\quad n \cdot(y - x)\leq 0\\
    & \Longleftrightarrow n \in N_K(x).
\end{align*}
Similarly 
\begin{align*}
    n \in N_{UK}(Ux) & \Longleftrightarrow \forall w \in UK,\quad n \cdot(w - Ux)\leq 0\\
    & \Longleftrightarrow \forall y \in K,\quad n \cdot(Uy - Ux)\leq 0\\
    &\Longleftrightarrow \forall y \in K,\quad U^{-1}n \cdot(y - x)\leq 0\\
    & \Longleftrightarrow U^{-1}n \in N_K(x)\\
    & \Longleftrightarrow n \in U N_K(x).\qedhere
\end{align*}
\end{proof}

\begin{defi}
Let $x \in \mathcal C$. The tangent hyperplanes to $K$ at $x$ are the hyperplanes
\[H_n := \left\{ y \in \R^d ~|~n \cdot y = n \cdot x  \right\}\]
for $n \in N_K(x) \setminus \{0\}$. If $n \cdot x > 0$, the \emph{normalized} equation of $H_n$ is  $\frac{n}{n \cdot x} \cdot y = 1$.
\end{defi}
In the smooth case, we recover the usual definition of the tangent space.

\begin{lemma}
\label{lem:existence_normalized_equation}
If $0$ lies in the interior of $K$, then every tangent hyperplane has a normalized equation, that is:
\[\forall x \in \mathcal C,\quad\forall n \in N_K(x) \setminus \{0\},\quad n \cdot x > 0.\]
\end{lemma}
\begin{proof}
Let $x \in \mathcal C$ and $n \in N_K(x) \setminus \{0\}$.
For $t>0$ small enough, $tn \in K$ because $0 \in \mathring{K}$, so we have $n \cdot (tn -x) \leq 0$, that is, $n \cdot x \geq t \|n\|^2 >0$.
\end{proof}

\begin{defi}
\label{def:dual_curve}
We assume that $0 \in \mathring{K}$. Let $x \in C$. The dual points to $K$ at $x$ are the elements of 
\begin{equation}
\label{eq:dual_points}
   x^* = \left\{ \frac{n}{n \cdot x} ~|~ n \in N_K(x)\setminus\{0\} \right\}. 
\end{equation}
In other words, it is the set of $n$ such that the normalized equations of tangent hyperplanes to $K$ at $x$ are the $\{n \cdot y = 1\}$, that is 
\begin{equation*}
%\label{eq:dual_points_bis}
x^* = \left\{ z \in N_K(x)~|~ z \cdot x = 1 \right\}.
\end{equation*}
The set $\mathcal{C}^* = \bigcup_{x \in C} x^*$ is the dual curve of $\mathcal{C}$.
\end{defi}

Let us do two remarks on Definition~\ref{def:dual_curve}. First, in the definition \eqref{eq:dual_points} of $x^*$, one can consider only \emph{unit} vectors $n$ of the normal cone, which we will often do in the proofs. Second, one can get rid of the condition $0 \in \mathring{K}$ to define the dual curve, but the dual curve is then defined in a projective space, which we do not need here.

From now on, we will always assume that $0 \in \mathring{K}$.

\begin{lemma}
\label{prop:compactness_dual}
The dual curve $\mathcal{C}^*$ is compact.
\end{lemma}
\begin{proof}
Let $r>0$ such that $\bar{B}(0,r) \subset K$. Let $z = \frac{u}{u \cdot x} \in \mathcal{C}^*$, where $x \in \mathcal{C}$ and $u \in N_K(x) \cap \mathbb S^{d-1}$. Then \[u \cdot x = u \cdot (ru + x-ru) = r + u \cdot (x-ru).\] But $ru \in K$ and $u \in N_K(x)$, so  $u \cdot (x - ru) \geq 0$, hence $u \cdot x \geq r$, i.e., $\left\| z \right\| \leq \frac{1}{r}$. We proved that $\mathcal{C}^* \subset \bar{B}\left(0, \frac{1}{r} \right)$, so $\mathcal{C}^*$ is bounded.

Let $z \in \overline{\mathcal{C}^*}$ and $\left( z_k \right)_{k\geq 0}$ a sequence of $\mathcal{C}^*$ that tends to $z$. For all $k\geq 0$, there exist $x_k \in \mathcal{C}$ and $u_k \in N_K(x_k) \cap \mathbb{S}^{d-1}$ such that $z_k = \frac{u_k}{u_k \cdot x_k}$. The sequence $\left(x_k, u_k\right)_{k\geq 0}$ lies in the compact set $\mathcal{C} \times \mathbb{S}^{d-1}$, so it has a subsequential limit $(x,u) \in \mathcal{C} \times \mathbb S^{d-1}$ and $z = \frac{u}{u\cdot x}$. Finally, taking the subsequential limit of the inequalities $u_k \cdot (y-x_k) \leq 0$ for $y \in K$ leads to $u \cdot (y-x) \leq 0$, hence $u \in N_K(x)$, and thus $z \in x^*$. Therefore, $z \in \mathcal{C}^*$, so $\mathcal{C}^*$ is closed.

In conclusion, $\mathcal{C}^*$ is a bounded, closed subset of $\R^d$, so it is compact.
\end{proof}

\begin{lemma}\label{lem:frontiere_convexe_interieur_non_vide}
Let $\mathsf{K}$ be a compact, convex subset of $\R^d$ such that $0 \in \mathring{\mathsf{K}}$. Let $z \in \partial \mathsf{K}$. Then for all $t >1$, $tz \notin \mathsf K$.
\end{lemma}
\begin{proof}
Let $r >0$ such that $B(0,r) \subset \mathsf K$.
We claim that if there exists $t >1$ such that $tz \in \mathsf K$, then $B\left(z,\left( 1-\frac{1}{t} \right) r\right) \subset \mathsf K$. Indeed, let $u \in \mathbb{S}^{d-1}$ and $0 \leq \varepsilon < \left( 1-\frac{1}{t} \right) r$. We have $z + \varepsilon u = \frac{1}{t} t z + \left( 1-\frac{1}{t} \right) \frac{\varepsilon}{1-\frac{1}{t}} u$, which is a convex combination of $tz$ and $\frac{\varepsilon}{1 - \frac{1}{t}}u$, with $tz \in \mathsf K$ and $\frac{\varepsilon}{1 - \frac{1}{t}}u \in \mathsf K$, so $z + \varepsilon u \in \mathsf K$, which proves $B\left(z,\left( 1-\frac{1}{t} \right) r\right) \subset \mathsf K$. This contradicts $z \in \partial \mathsf K$. 
\end{proof}

\begin{prop}
\label{prop:duality_preserves_convexity}
Duality preserves convexity, that is, there exists a compact convex subset $K^*$ of $\R^d$ such that  $0 \in \mathring{K^*}$ and $\mathcal{C}^* = \partial (K^*)$.
The set $K^*$ is called the dual set of $K$.
\end{prop}
\begin{proof}
We define $K^*$ as the convex hull of $\mathcal{C}^*$. Since $\mathcal{C}^*$ is a compact subset of $\R^d$ according to Proposition~\ref{prop:compactness_dual}, its convex hull $K^*$ is compact. 

Let $u \in \mathbb{S}^{d-1}$. By item~\ref{it:lem:normal_cones_3} of Lemma~\ref{lem:normal_cones}, there exists $x \in \mathcal{C}$ such that $u \in N_K(x)$. Therefore, $\mathcal{C}^*$ contains a positive multiple $r_u u$ of $u$. Applying this to the vectors $\left(\pm e_i\right)_{1 \leq i \leq d }$, where $(e_1, \ldots, e_d)$ is the canonical basis of $\R^d$, we obtain that $K^*$ contains the convex hull of $\{\pm r_{\pm e_i} e_i\}$, which contains the ball centered in $0$ of radius $\min_{i}r_{\pm e_i}$ for the norm $\|\cdot\|_1$. Therefore, $0$ lies in the interior of $K^*$.

Let $z \in \partial(K^*)$. We prove that $z \in \mathcal{C}^*$. The preceding result shows that $z \neq 0$. Therefore, according to item~\ref{it:lem:normal_cones_3} of Lemma~\ref{lem:normal_cones}, there exists $x\in \mathcal{C}$ such that $z \in N_K(x)$. To get $z \in \mathcal{C}^*$, all that remains to be proved is that $z \cdot x = 1$. We start by proving that $z \cdot x \leq 1$.
By definition of $K^*$, we can write $z = \sum_{i=0}^{d} \lambda_i z_i$ where $\lambda_0, \ldots, \lambda_{d} \geq 0$, $\sum_{i=0}^d \lambda_i = 1$ and $z_0, \ldots, z_d \in \mathcal{C}^*$. Writing each of the $z_i$ as $z_i = \frac{u_i}{u_i \cdot x_i}$ where $x_i \in \mathcal{C}$ and $u_i \in N_{K}(x_i)$, we get \[z = \sum_{i=0}^{d} \lambda_i \frac{u_i}{u_i \cdot x_i}.\]
But $u_i \in N_K(x_i)$ and $x \in K$, so $u_i \cdot (x-x_i) \leq 0$, i.e., $\frac{u_i \cdot x}{u_i \cdot x_i} \leq  1$. Summing up, we get
\[z \cdot x = \sum_{i=0}^d \lambda_i \frac{u_i \cdot x}{u_i \cdot x_i} \leq \sum_{i=0}^d \lambda_i = 1.\]
Let us move to the second inequality, $z \cdot x \geq 1$. Assume by contradiction that $z \cdot x < 1$. Then $\frac{z}{z\cdot x} \in \mathcal{C}^* \subset K^*$. This contradicts Lemma~\ref{lem:frontiere_convexe_interieur_non_vide} with $\mathsf K = K^*$ and $t = \frac{1}{z \cdot x}$. In conclusion, $z \cdot x = 1$ so $z \in \mathcal{C}^*$, which proves $\partial (K^*) \subset \mathcal{C}^*$.

Conversely, let $z \in \mathcal{C}^*$. It means that there exists $x \in \mathcal{C}$ such that $z \in N_K(x)$ and 
\begin{equation}
\label{eq:duality_preserves_convexity}
    z \cdot x = 1.
\end{equation}
We know that $z \in K^*$. If $z \notin \partial \left(K^*\right)$, then there exists $t>1$ such that $t z \in \partial \left(K^*\right)$. Using the first inclusion, we get $tz \in \mathcal{C^*}$, which means that there exists $\widetilde{x} \in \mathcal{C}$ such that $tz \in N_K\left( \widetilde{x}\right)$ and $tz \cdot \widetilde{x} = 1$. Because $x \in K$ and $tz \in N_K\left(\widetilde{x}\right)$, we have $tz\cdot \left( x - \widetilde{x} \right) \leq 0$, that is $tz \cdot x \leq tz \cdot \widetilde{x} = 1$. Therefore, $z \cdot x \leq \frac{1}{t} < 1$, which contradicts \eqref{eq:duality_preserves_convexity}. In conclusion, $z \in \partial(K^*)$ and the proof is complete.
\end{proof}

The dual set $K^*$ satisfies the same assumptions as $K$, so $\mathcal{C}^*$ has a dual curve $\left(\mathcal{C}^*\right)^*$, which we will simply write $\mathcal{C}^{**}$. 

\begin{prop}
Duality is an involution, i.e.,  $\mathcal{C}^{**} = \mathcal{C}$.
\end{prop}
\begin{proof}
Since $K$ is a compact, convex set such that $0 \in \mathring{K}$, Lemma~\ref{lem:frontiere_convexe_interieur_non_vide} ensures that for every $u \in \mathbb{S}^{d-1}$, there is a unique $r_u >0$ such that $r_u u \in \partial K = \mathcal{C}$, and $\mathcal{C} = \left\{r_u u~|~u \in \mathbb{S}^{d-1} \right\}$. Similarly, for every $u \in \mathbb{S}^{d-1}$, there is a unique $r_u^{**} >0$ such that $r_u^{**} u \in \mathcal{C}^{**}$, and $\mathcal{C}^{**} = \left\{r_u^{**} u~|~u \in \mathbb{S}^{d-1} \right\}$. Therefore, if we prove an inclusion between $\mathcal{C}^{**}$ and $\mathcal{C}$, we will have $r_u = r^{**}_u$ for every $u \in \mathbb{S}^{d-1}$ and thus equality between the two curves.

Let $x \in \mathcal{C}$. Let $z \in \mathcal{C}^*$ such that $z \in N_K(x)$ and $z \cdot x = 1$ (such a $z$ exists: we can chose $w \in N_K(x) \setminus \{0\}$ thanks to Lemma~\ref{lem:normal_cones} and since $w \cdot x >0$ according to Lemma~\ref{lem:existence_normalized_equation}, we only need to set $z = \frac{w}{w\cdot x}$). To get $x \in \mathcal{C}^{**}$, it remains to show that $x \in N_{K^*}(z)$, that is 
\begin{equation}
\label{eq:duality_involution}
\forall \widetilde{z} \in K^*,\quad x \cdot \widetilde{z} \leq x \cdot z = 1.
\end{equation}
Every $\widetilde{z} \in K^*$ is a convex combination of elements of $\mathcal{C}^*$, so it is enough to prove \eqref{eq:duality_involution} when $\widetilde{z} \in \mathcal{C}^*$. In that case, we can write $\widetilde{z} = \frac{u}{u \cdot y}$ with $y \in \mathcal{C}$ and $u \in N_K(y) \setminus \{0\}$, so $x \cdot \widetilde{z} = \frac{u \cdot x}{u \cdot y}$. But $x \in K$ and $u \in N_K(y)$, so $u \cdot x \leq u \cdot y$, i.e., $\frac{u \cdot x}{u \cdot y} \leq 1$, i.e., $x \cdot \widetilde{z} \leq 1$. In conclusion, we showed \eqref{eq:duality_involution}, hence $x \in \mathcal{C}^{**}$. As mentioned at the beginning of the proof, the inclusion $\mathcal{C} \subset \mathcal{C}^{**}$ is enough to conclude.
\end{proof}

\begin{prop}
\label{prop:duality_exchange_inclusion}
Let $K_1, K_2$ be two compact, convex subsets of $\R^d$, with $0 \in \mathring{K}_1 \cap \mathring{K}_2$, such that $K_1 \subset K_2$. Then $K_2^* \subset K_1^*$.
\end{prop}
\begin{proof}
It is enough to prove that $\mathcal{C}_2^* \subset K_1^*$. Let $z_2 = \frac{u_2}{u_2 \cdot x_2} \in \mathcal{C}_2^*$, where $x_2 \in \mathcal{C}_2$ and $u_2 \in N_{K_2}(x_2) \cap \mathbb{S}^{d-1}$. By item~\ref{it:lem:normal_cones_3} of Lemma~\ref{lem:normal_cones}, there exists $x_1 \in \mathcal{C}_1$ such that $u_2 \in N_{K_1}(x_1)$, hence $\frac{u_2}{u_2 \cdot x_1} \in K_1^*$. But $x_1 \in K_1 \subset K_2$ and $u_2 \in N_{K_2}(x_2)$, so $u_2 \cdot (x_1 - x_2) \leq 0$, i.e., $u_2 \cdot x_1 \leq u_2 \cdot x_2$. Therefore \[ z_2 = \frac{u_2}{u_2 \cdot x_2} \in \left[ 0, \frac{u_2}{u_2 \cdot x_1} \right] \subset K_1^*.\qedhere\]
\end{proof}

\begin{prop}\label{prop:dual_homotheties_rotations}
Let $K$ be a compact, convex subset of $\R^d$ such that $0 \in \mathring{K}$, $t >0$ and $U \in O_d(\R)$. Then $(tK)^* = \frac{1}{t} K ^*$ and $(UK)^* = U(K^*)$.
\end{prop}
\begin{proof}
Let $z \in \R^d$.
We have 
\begin{align*}
    z \in (tK)^* &\Longleftrightarrow \exists x \in \mathcal{C},\quad \exists u\in N_{tK}(tx),\quad  z = \frac{u}{u \cdot (tx)}\\
    &\Longleftrightarrow \exists x \in  \mathcal{C},\quad \exists u\in N_{K}(x),\quad  z = \frac{u}{u \cdot (tx)} = \frac{1}{t} \frac{u}{u \cdot x}\\
    & \Longleftrightarrow z \in \frac{1}{t}K^*,
\end{align*}
where the second equivalence used the equality $N_{tK}(tx) = N_K(x)$ of Lemma~\ref{lem:normal_cones}.
Similarly,
\begin{align*}
    z \in (UK)^* &\Longleftrightarrow \exists x \in \mathcal{C},\quad \exists u\in N_{UK}(Ux),\quad  z = \frac{u}{u \cdot Ux}\\
    &\Longleftrightarrow \exists x \in  \mathcal{C},\quad \exists 
    v \in N_{K}(x),\quad  z = \frac{Uv}{Uv \cdot Ux} = U \frac{v}{v \cdot x}\\
    & \Longleftrightarrow z \in U(K^*), 
\end{align*}
where the second equivalence uses the equality $N_{UK}(Ux) = UN_K(x)$ of Lemma~\ref{lem:normal_cones}.
\end{proof}

\begin{prop}
\label{prop:convergence_duality}
Let $K$ and $\left( K_k \right)_{k\geq 0}$ be compact, convex subsets of $\R^d$ such that $0 \in \mathring{K}$ and $0 \in \mathring{K}_k$ for all $k\geq 0$. We assume that there exist sequences $\left(\varepsilon^-_k\right)_{k\geq 0}$ and $\left(\varepsilon^+_k\right)_{k\geq 0}$ of positive numbers that tend to $0$ such that for every $k\geq 0$, \[\left(1 - \varepsilon^-_k\right) K \subset K_k \subset \left(1+ \varepsilon^+_k\right) K.\]
Then, uniformly in the sense of Definition~\ref{def:cv_sets_limit_shape}, and consequently for the Hausdorff distance,
\[K_k \xrightarrow[k \to + \infty]{}K \quad \text{and}\quad K_k^* \xrightarrow[k \to + \infty]{}K^*.\]
\end{prop}
\begin{proof}
Let us first prove the result for $K_k$. The assumptions on the convex sets imply that there is a bounded family $\left( r_u \right)_{u \in \mathbb{S}^{d-1}}$ such that $K = \left\{ r u~|~ u \in \mathbb{S}^{d-1},0 \leq r \leq r_u \right\}$. Therefore, using the notations of Definition~\ref{def:cv_sets_limit_shape}, the inclusions $\left(1 - \varepsilon^-_k\right) K \subset K_k \subset \left(1+ \varepsilon^+_k\right) K$ allow us to set \[r_{u,k} = \left(1- \varepsilon^-_k\right) r_u\quad \text{and}\quad R_{u,k} = \left(1 + \varepsilon^+_k\right)r_u.\] Therefore, \[ \|r_{u,k} - r_u\|_\infty =   \varepsilon^-_k \sup_{u \in \mathbb{S}^{d-1}}r_u \xrightarrow[k \to +\infty]{}0 \quad\text{and}\quad\|R_{u,k} - r_u\|_\infty =   \varepsilon^+_k \sup_{u \in \mathbb{S}^{d-1}}r_u \xrightarrow[k \to +\infty]{}0,\]
hence the uniform convergence of $K_k$.

For the duals, we use Propositions~\ref{prop:duality_exchange_inclusion}~and~\ref{prop:dual_homotheties_rotations} to obtain \[ \frac{1}{1+\varepsilon^+_k} K^* \subset K_k^* \subset \frac{1}{1-\varepsilon^-_k} K^* \] and use the first case with $\widetilde{\varepsilon}_k^- = 1 - \frac{1}{1+\varepsilon^+_k}$ and $\widetilde{\varepsilon}_k^+ = \frac{1}{1-\varepsilon^-_k} - 1$.
\end{proof}

\begin{prop}[Ellipsoid]\label{prop:dual_ellipsoid}
Let $\sigma$ be a symmetric, positive-definite matrix and \[K_\sigma = \left\{ x \in \R^d ~|~x^\intercal \sigma x \leq 1 \right\}.\] Then $K_\sigma^* = K_{\sigma^{-1}}$.
\end{prop}
\begin{proof}
When $\sigma$ is diagonal, it is a simple computation.
In the general case, we write $\sigma = U^\intercal D U $ where $D$ is diagonal and $U \in O_d(\R)$. Then 
\begin{equation}
\label{eq:dual_ellipsoid}
    K_\sigma = \left\{ x \in \R^d ~|~(Ux)^\intercal D (Ux) \leq 1 \right\} = U^{-1}K_D,
\end{equation}
so Proposition~\ref{prop:dual_homotheties_rotations} yields
\[K_\sigma^* = U^{-1}(K_D^*).\]
Using the diagonal case and \eqref{eq:dual_ellipsoid} with $\sigma^{-1} = U^T D^{-1}U$ instead of $\sigma$, we finally get 
\[K_\sigma^* = U^{-1}(K_D^*) = U^{-1}K_{D^{-1}} = K_{\sigma^{-1}}.\qedhere\]
\end{proof}

\begin{prop}[Polytope]
\label{prop:dual_polytope}
Let $x_1, \ldots, x_n \in \R^d$ and \[K = \bigcap_{i \in \{1, \ldots, n\}} \left\{ s \in \R^d~|~ s \cdot x_i \leq 1 \right\}.\]
The set $K$ is compact if and only if $0$ lies in the interior of the convex hull of $\{x_1, \ldots, x_n\}$. In that case, this convex hull is the dual convex $K^*$ of $K$.
\end{prop}
\begin{proof}
The convex $K$ is clearly a closed subset of $\R^d$, so it is compact if and only if it is bounded, which comes down to saying that for all $u \in \mathbb S^{d-1}$, there exists $i \in \{1, \ldots, n\}$ such that $u \cdot x_i >0$. If $0$ lies in the interior of the convex hull of $\{x_1, \ldots, x_n\}$, then for $\varepsilon >0$ small enough, $\varepsilon u$ can be written as a convex combination of $\{x_1, \ldots, x_n\}$. Since $u \cdot \varepsilon u >0$, there has to be at least one $x_i$ such that $u \cdot x_i >0$, so $K$ is bounded. Conversely, if $K$ is compact, we will obtain that $0$ lies in the interior of the convex hull of $\{x_1, \ldots, x_n\}$ as a consequence of the computation of $K^*$ and of Proposition~\ref{prop:duality_preserves_convexity}, which ensures that $0$ lies in the interior of $K^*$.

We assume that $K$ is compact. To compute $K^*$, we first note that \[K = \bigcap_{x \in \mathrm{extr}(\{x_1, \ldots, x_n\})} \left\{ s \in \R^d~|~s \cdot x \leq 1 \right\},\]because if one of the $x_i$ is a convex combination of the others, then the condition $s\cdot x_i\leq 1$ is redundant. The faces of the polytope $K$ are thus delimited by the hyperplanes $H_x = \left\{ s \in \R^d~|~s \cdot x = 1 \right\}$ for $x \in \mathrm{extr}(\{x_1, \ldots, x_n\})$. At a point $y$ that lies in the interior of the face delimited by $H_x$, the normal cone is $\R_{\geq 0}x$, so the dual point associated with $y$ is $x$. In particular, $K^*$ contains $\mathrm{extr}(\{x_1, \ldots, x_n\})$. Similarly, if $y$ is at the intersection of the hyperplanes $H_{x_i}$, $i\in I$, then the normal cone at $y$ consists of convex combinations of the $x_i$, $i \in I$, so $K^*$ is a subset of the convex hull of $\mathrm{extr}(\{x_1, \ldots, x_n\})$. In conclusion, $K^*$ is a subset of the convex hull of $\mathrm{extr}(\{x_1, \ldots, x_n\})$ and it contains $\mathrm{extr}(\{x_1, \ldots, x_n\})$. By convexity of $K^*$, it means that $K^*$ is the convex hull of $\mathrm{extr}(\{x_1, \ldots, x_n\})$, that is the convex hull of $\{x_1, \ldots, x_n\}$.
\end{proof}

\section{An interpretation in terms of amoebas}
\label{app:amoebas}

In this appendix, we assume that $\mu$ has finite support.
We give a geometric interpretation of the Doob transform and
the set $\partial\mathcal{T}$ with tools from real algebraic geometry.

A quantity directly related to the real Laplace transform $L\mu$ is the Fourier
transform of $\mu$, denoted by $\mathcal{F}\mu\colon(\mathbb{C}^*)^d\rightarrow
\mathbb{C}^*$, defined by:
\begin{equation*}
  \forall
  z=(z_1,\ldots,z_d)\in(\mathbb{C}^*)^d,\quad
  \mathcal{F}\mu(z):=
  \sum_{x\in\mathbb{Z}^d} z^x \mu_{i,j}(x),
\end{equation*}
where $z^x$ stands for $\prod_{i=1}^d z_i^{x_i}$.
With the convention that $e^t = (e^{t_1},\ldots,e^{t_d})$ for
$t\in\mathbb{R}^d$,
we have the following relation:
\begin{equation*}
  \mathcal{F}\mu(e^t) = L\mu(t).
\end{equation*}
The Fourier transform $\mathcal{F}\mu$ is a $d\times d$ matrix, whose entries are Laurent polynomials
in $z_1,\ldots,z_d$. The \emph{characteristic polynomial} of the random walk is
\begin{equation*}
P(z) = \det\bigl(I_d - \mathcal{F}\mu(z)\bigr),
\end{equation*}
and its \emph{spectral hypersurface} is the algebraic variety in
$(\mathbb{C}^*)^d$ consisting of the zero set of $P$:
\begin{equation*}
  \mathscr{C} = \{ z\in(\mathbb{C}^*)^d \ |\ P(z)=0\}.
\end{equation*}
In the special case when $d=2$, the spectral hypersurface is an algebraic curve,
and we rather use the term \emph{spectral curve}. This terminology is directly
borrowed from dimer models~\cite{KOS} and spanning trees~\cite{Kenyon:forests},
but has been used for much longer in exactly solvable models in statistical
mechanics, and integrable systems.

An interesting object to apprehend the geometry of this spectral hypersurface is
its so-colled \emph{amoeba}~$A(\mathscr{C})$, defined as:
\begin{equation*}
  A(\mathscr{C}) = \{\operatorname{Log}(z)\ ;\ z\in \mathscr{C}\}
  \subset \mathbb{R}^d,
\end{equation*}
where $\operatorname{Log}(z)=(\log|z_1|,\ldots,\log|z_d|)$. See
for example~\cite[Chap.~6]{Gelfand1994} for more details.

Gershgorin's theorem guarantees that when
$|z|=(|z_1|,\ldots,|z_d|)=(1,\dots,1)$, $\mathcal{F}\mu(z)$ has no
eigenvalue near 1. This implies that the point $(0,\ldots, 0)$ is in the
complement of $A(\mathscr{C})$.

When $d=2$, and the graph of allowed transitions between vertices for the
(killed)\ walk on $\mathbb{Z}^2\times\{1,\ldots,p\}$ can be realized as a planar,
biperiodic graph, then the spectral curve is a \emph{Harnack
curve}~\cite{Mikhalkin,Kenyon:forests}. In particular, the boundary of the
amoeba is given by the image by $\operatorname{Log}$ of the real locus of $\mathscr{C}$, that is, the
points $z=(z_1,\ldots,z_d)$ of $\mathscr{C}$ for which all $z_i$ are real.
This is not true in general. However, the Perron-Frobenius theorem implies the
following result:

\begin{lemma}
  The boundary of the connected component of the complement of $A(\mathscr{C})$
  containing the origin is the image by $\operatorname{Log}$ of a connected component
  of $\mathscr{C}^+$, the real positive locus of $\mathscr{C}$:
  \begin{equation*}
    \mathscr{C}^+ = \{z \in (\mathbb{R}_{>0})^d\ ;\ P(z)=0\}.
  \end{equation*}
\end{lemma}

\begin{proof}
  By general results of real algebraic geometry~\cite{Mikhalkin}, each connected
  component of the complement of the amoeba are convex.
  Let $r=(r_1,\ldots,r_d)$ a point on the boundary of that connected component.
  Then the whole interval $[0,r)$ is contained in the connected component of the
  complement.
  Define $M(s)=\mathcal{F}\mu(e^{sr_1},\ldots,e^{sr_d})$ for $s\in[0,1]$. This
  matrix has non-negative entries, so Perron-Frobenius theorem applies to
  $M(s)$, which should have spectral radius given by its largest positive
  eigenvalue $\lambda(s)$. The point $s=1$ is the first point where
  $\det(\operatorname{I}_d-\mathcal{F}\mu(z))=0$, for $|z_j| = e^{s r_j}$, and
  the maximum principle implies that the only $z$ with this constraint
  satisfying this equation is $z=e^{s r}$. The eigenvalue which is then equal to
  one must be the largest one, otherwise, by the intermediate value theorem, the
  Perron-Frobenius eigenvalue would have reached one for some $s<1$, which would
  be in contradiction with the definition of $r$.
\end{proof}

This allows one to make a connection between the geometry of the amoeba
$A(\mathscr{C})$ and objects introduced earlier.
\begin{itemize}
  \item The closure of the connected component of the complement of $A(\mathscr{C})$
    containing the origin is the set $\mathcal{T}$ introduced in Definition~\ref{def:rho_T}.
  \item The fact that starting from the origin, every direction $u\in\mathbb{S}^{d-1}$ corresponds bijectively to a point of $\mathcal{T}$ is
    a reformulation of the Ney-Spitzer homeomorphism in this context (see \cite{BoRa-22} for related ideas).
  \item For every $t\in\partial\mathcal{T}$, the eigenvector $\varphi_t$ is a non-zero vector
    in the kernel of the matrix $\operatorname{I}_d-\mathcal{F}\mu(z)$ for $z=(e^{t_1},\ldots,e^{t_d})=e^t$.
  \item Performing the Doob transform with $\varphi_t$ to obtain $\mu_t$ implies that for any $w\in (\mathscr{C}^*)^d$,
    \begin{equation*}
        \mathcal{F}\mu_t (w) = \Phi_t^{-1}\cdot \mathcal{F}\mu_t(e^{t}w) \cdot \Phi_t,
    \end{equation*}
    where $\Phi_t$ is the diagonal matrix whose diagonal entries are the values of $\varphi_t$, and the product $e^t w$ is understood coordinatewise.
  \item In particular the last point implies that the amoeba $A(\mathscr{C}_t)$ for $\mu_t$ is obtained from the one for $\mu$ by a translation of vector $-t$, which means that the origin is now on the boundary of $A(\mathscr{C}_t)$.
  \item The normal to $\mathcal{T}$ at $t$, which is parallel to $\nabla\rho(t)$, can be interpreted as the direction of the drift for the conditioned walk, as noted in~\cite[Def.~1.3, Prop.~2.16]{Ba-24}.
\end{itemize}

\section*{Acknowledgments}
TB, CB and KR are supported by the project DIMERS (\href{https://anr.fr/Project-ANR-18-CE40-0033}{ANR-18-CE40-0033}). TB and KR are supported by the project RAWABRANCH (\href{https://anr.fr/Project-ANR-23-CE40-0008}{ANR-23-CE40-0008}), funded by the French National Research Agency, and by the France 2030 program Centre Henri Lebesgue  (\href{https://anr.fr/ProjetIA-11-LABX-0020}{ANR-11-LABX-0020-01}). SM was partly supported by Simons Foundation Grant No. 422190. Part of this research was performed while CB was visiting the Institute for Pure and Applied Mathematics (IPAM), which is supported by the National Science Foundation (Grant No.\ \href{https://www.nsf.gov/awardsearch/showAward?AWD_ID=1925919}{DMS-1925919}). 
TB and KR thank
the VIASM (Hanoï, Vietnam)\ for their hospitality and wonderful working conditions. For part of this research SM was visiting Yerevan State University through the support of the Fulbright Scholar program, and would like to thank them for their hospitality.  
The authors would like to thank Ian Alevy and Sanjay Ramassamy for very interesting discussions at an early stage of the project.

\bibliographystyle{plain} % We choose the "plain" reference style

\end{document}